\documentclass[12pt]{article}
\usepackage{amssymb,amsmath,amsthm,bm,setspace,mathrsfs,paralist}
\usepackage{url,color,units,dsfont}
\RequirePackage[colorlinks,citecolor=blue,urlcolor=purple,linkcolor=purple]{hyperref}

\usepackage[normalem]{ulem}

\DeclareMathOperator{\Var}{Var}
\DeclareMathOperator{\Cov}{Cov}

\newcommand{\N}{\mathds{N}}
\newcommand{\Z}{\mathds{Z}}
\newcommand{\Q}{\mathds{Q}}
\newcommand{\R}{\mathds{R}}

\newcommand{\cB}{\mathcal{B}}
\newcommand{\cT}{\mathcal{T}}

\newcommand{\<}{\langle}
\renewcommand{\>}{\rangle}

\renewcommand{\P}{\mathrm{P}}
\newcommand{\E}{\mathrm{E}}

\renewcommand{\d}{{\rm d}}

\newcommand{\e}{{\rm e}}
\renewcommand{\geq}{\geqslant}
\renewcommand{\leq}{\leqslant}
\renewcommand{\ge}{\geqslant}
\renewcommand{\le}{\leqslant}

\usepackage[dvipsnames]{xcolor}
\newenvironment{Note-by-Le}{\par\color{BrickRed}}{\par}

\author{Le Chen\\Univ.\  Kansas\and Jingyu Huang\\Univ.\ Utah
	\and D. Khoshnevisan\\Univ.\  Utah
	\and Kunwoo Kim\\POSTECH}
\title{\bf Dense blowup for parabolic SPDEs\thanks{
	Research supported in part by the NSF grants DMS-1307470 
	and DMS-1608575 [D.K.], and DMS-1440140
	 [J.H., D.K., and K.K.] while three of the authors
	were in residence at the Mathematical Sciences Research
	Institute at UC Berkeley in Fall of 2015.}}
\date{February 26, 2017}

\newtheorem{stat}{Statement}[section]
\newtheorem{proposition}[stat]{Proposition}
\newtheorem{corollary}[stat]{Corollary}
\newtheorem{theorem}{Theorem}
\newtheorem{lemma}[stat]{Lemma}
\theoremstyle{definition}

\newtheorem{remark}[stat]{Remark}
\newtheorem{conj}{Conjecture}
\newtheorem*{OP}{Open Problem}

\numberwithin{equation}{section}

\usepackage[margin=1in]{geometry}
\begin{document}
\maketitle
\begin{abstract}
	The main result of this paper is that there are examples of stochastic partial differential equations
	[hereforth, SPDEs] of the type 
	\[
		\partial_t u=\tfrac12\Delta u +\sigma(u)\eta
		\qquad\text{on $(0\,,\infty)\times\R^3$}
	\]
	such that the solution exists and is unique as a random field
	in the sense of Dalang \cite{Dalang} and Walsh \cite{Walsh},
	yet the solution has unbounded oscillations in every open neighborhood
	of every space-time point. We are not aware of the existence of
	such a construction in spatial dimensions below $3$.
	
	En route, it will be proved that
	there exist a large family of parabolic SPDEs whose
	moment Lyapunov exponents grow at least sub exponentially in its
	order parameter in the sense that there exist $A_1,\beta\in(0\,,1)$ such that
	\[
		\underline{\gamma}(k) := 
		\liminf_{t\to\infty}t^{-1}\inf_{x\in\R^3}
		\log\E\left(|u(t\,,x)|^k\right) \ge A_1\exp(A_1 k^\beta)
		\qquad\text{for all $k\ge 2$}.
	\]
	This sort of ``super intermittency'' is combined with a local linearization of the solution,
	and with techniques from Gaussian analysis
	in order to establish the unbounded oscillations of the sample functions
	of the solution to our SPDE.\\
	
	\noindent{\it Keywords.} Stochastic partial differential equations,  blowup,
		intermittency.\\
	\noindent{\it \noindent AMS 2010 subject classification.}
	Primary 35R60, 60H15; Secondary 60G15.
\end{abstract}

\section{Introduction}
Throughout, let us choose and fix a non random, globally Lipschitz-continuous function 
$\sigma:\R\to\R$, and consider the stochastic heat equation,
\begin{equation}\label{SHE}
	\frac{\partial u(t\,,x)}{\partial t} = \frac12 (\Delta u)(t\,,x) + \sigma(u(t\,,x))\eta(t\,,x)
	\qquad\text{for $(t\,,x)\in(0\,,\infty)\times\R^3$},
\end{equation}
subject to initial value $u(0)\equiv1$. The forcing term $\eta$ is a white noise with
homogeneous correlations in its spatial variable; that is, $\eta$ is a centered,
generalized Gaussian random field with 
\[
	\Cov[\eta(t\,,x)\,,\eta(s\,,y)] = \delta_0(t-s)f(x-y)\qquad
	\text{for all $(t\,,x),(s\,,y)\in\R_+\times\R^3,$}
\]
where the  \emph{spatial correlation function} $f:\R^3\to\R_+$ is a non random,
non-negative, tempered, and positive semi-definite function. In principle, such equations can be ---
and have been --- studied on $\R_+\times\R^n$ for any integer $n\ge 1$. We will
soon explain why we study them for $n=3$ here.

Let $\widehat{g}$ denote the Fourier transform of any distribution $g$ on $\R^n$,
normalized so that 
\[
	\widehat{g}(z) = \int_{\R^3} \e^{ix\cdot z}g(x)\,\d x
	\qquad\text{for all $z\in\R^3$ and $g\in L^1(\R^3)$}.
\]
The starting point of this article is the following existence and uniqueness theorem of
Dalang \cite{Dalang}. Recall that $\widehat{f}\ge0$ almost everywhere because $f$
is positive semi-definite.

\begin{theorem}[Dalang \protect{\cite{Dalang}}]\label{th:Dalang}
	If
	\begin{equation}
		\int_{\R^3}\frac{\widehat{f}(z)}{1+\|z\|^2}\,\d z<\infty,
		\label{cond:Dalang}
	\end{equation}
	then \eqref{SHE} has a random field solution $u$. Moreover, $u$ is unique subject to the condition that
	\[
		\sup_{t\in[0,T]}\sup_{x\in\R^3}\E\left(|u(t\,,x)|^k\right)<\infty
		\quad\text{for all $k\in[2\,,\infty)$}.
	\]
\end{theorem}

According to a general form of Doob's separability theorem \cite[Theorem 2.2.1, Chapter 5]{MPP},
we may --- and will --- tacitly assume without loss of generality that the 4-parameter process
$u$ is separable.

Dalang \cite{Dalang} has observed that 
Condition \eqref{cond:Dalang} is also necessary
in the case that $\sigma$ is identically a constant.

Recall that the \emph{oscillation function} of a function
$\psi:\R^3\to\R$ is defined as
\[
	\text{\rm Osc}_\psi(x) := \lim_{\varepsilon\downarrow0}
	\sup_{a,b\in B(x,\varepsilon)}\left| \psi(a)-\psi(b) \right|
	\qquad\text{for all $x\in\R^3$},
\]
where
\begin{equation}\label{ball}
	B(x\,,\varepsilon):=\{y\in\R^3:\, \|y-x\|<\varepsilon\}
	\qquad\text{for all $x\in\R^3$ and $\varepsilon>0$}.
\end{equation}
The main results of this paper are the following two theorems.
In one form or another, the next two theorems show the existence of models of
\eqref{SHE} that can have unbounded oscillations everywhere. This holds despite
the fact that $u(t\,,x)$ is a finite random variable at all non random 
space-time points $(t\,,x)\in(0\,,\infty)\times\R^3$.

\begin{theorem}\label{th:conditional}
	Suppose in addition that $\sigma^{-1}\{0\}=\{0\}$ and $\sigma$ is bounded. Then,
	there exist correlation functions $f:\R^3\to\R_+$ that satisfy 
	\eqref{cond:Dalang} and
	\[
		\P  \{\text{\rm Osc}_{u(t)}(x)=\infty \mid u(t\,,x) \neq 0\}=1
		\qquad\text{for every $(t\,,x)\in(0\,,\infty)\times\R^3$}.  
	\]
\end{theorem}

This sort of extremely bad behavior of SPDEs has been observed earlier
only for simpler, constant-coefficient SPDEs \cite{DL1,DL2,FK2}
and/or exactly-solvable ones \cite[Theorem 1.2]{MP} that are forced by
``very wild,'' non-Gaussian noise terms. We believe that
the methods of the present paper are novel, in addition to being general enough to include
a variety of nonlinear SPDEs that are driven by Gaussian white-noise forcing terms. For
a non-trivial variation of Theorem \ref{th:conditional}, see Theorem \ref{th:local} below.

Before we describe that variation,  we first would like to explain why 
we consider equations on $\R_+\times\R^n$
only when $n=3$: Spatial dimension three is the smallest dimension
in which we know how to establish the blowup results of
Theorem \ref{th:conditional} and the next theorem.

\begin{theorem}\label{th:local}
	If $0<\inf_{z\in\R}\sigma(z) \le \sup_{z\in\R}\sigma(z)<\infty$,
	then there exist correlation functions $f:\R^3\to\R_+$ that satisfy 
	\eqref{cond:Dalang} and
	\[
		\P  \left\{ \text{\rm Osc}_{u(t)}(x)=\infty\right\}=1
		\qquad\text{for every $(t\,,x)\in(0\,,\infty)\times\R^3$}.
	\]
\end{theorem}

Theorems \ref{th:Dalang}, \ref{th:conditional}, and \ref{th:local} together imply that
there are models of $f$ that satisfy \eqref{cond:Dalang} such that,
for every $t>0$ fixed, the random function $u(t):\R^3\to\R$ has discontinuities of
the second kind. These theorems, particularly Theorem \ref{th:local}, fall short
of establishing the following conjectures.

\begin{conj}\label{Conj1}
	Under the hypotheses of Theorem \ref{th:conditional}, 
	there exist correlation functions $f:\R^3\to\R_+$ that satisfy 
	\eqref{cond:Dalang} and
	\begin{equation}\label{strong}
		\P\left\{ \text{\rm Osc}_{u(t)}(x)=\infty\text{ for all $(t\,,x)\in(0\,,\infty)\times\R^3$}\right\}=1.
	\end{equation}
\end{conj}

\begin{conj}\label{Conj2}
Suppose $\sigma(z)=z$ for all $z\in \R$. Then, there exist correlation functions $f : \R^3 \to \R_+$ that satisfy \eqref{cond:Dalang} and \eqref{strong} holds. 
 
\end{conj}  
The methods of this paper are efficient enough to prove Conjecture \ref{Conj1}
provided that the answer to the following is ``yes'':
\begin{OP}
	{ Under the hypotheses of Theorem \ref{th:conditional}. 
	}Is it true that 
	\[
		\P\left\{ u(t\,,x) >0  \text{ for all rational $t\ge0$ and $x\in\Q^3$}\right\}=1?
	\]
\end{OP}
The only strict positivity type of theorem for SPDEs that we
are aware of is the celebrated result of Mueller \cite{Mueller1};
see also \cite[pp.\ 134--135]{Mueller2}. But that result,
and its proof, rely crucially on the {\it a priori} H\"older continuity of the solution.
This is a luxury that we do not have in the present setting, as is
corroborated by Theorems \ref{th:conditional}
and \ref{th:local}. The best-known result, along these lines, is the following
consequence of Corollary 1.2 of Chen and Huang \cite{CH}.

\begin{theorem}[Chen and Huang \protect{\cite[Corollary 1.2]{CH}}]\label{thm:nonnegative}
	If, additionally, $\sigma(0)=0$, then 
	\[
		\P\left\{ u(t\,,x) \ge 0  \text{ for all rational $t\ge0$ and $x\in\Q^3$}\right\}=1.
	\]
\end{theorem}

The remainder of this paper is devoted to proving Theorem \ref{th:conditional}.
At the end of the paper, we have also included a paragraph which outlines how
one can prove Theorem \ref{th:local} from Theorem \ref{th:conditional}.
In anticipation of those arguments
let us conclude the Introduction by introducing more notation that will
be used throughout the paper.

Throughout, let  $p_t(x) = p(t\,,x)$ denote the heat kernel in $\R^3$; that is,
\begin{equation}\label{p}
	p_t(x) := (2\pi t)^{-3/2} \e^{ -\|x\|^2/(2t)}
	\qquad\text{for all $x\in\R^3$ and $t>0$}.
\end{equation}
In particular, $p_t$ does not refer to the time derivative of the heat kernel;
rather the heat kernel itself.

We will use the following notation for shorthand. For any two functions $A,B:R\to\R$,
where $R$ is a topological space:
\begin{compactitem}
\item $A(r)\sim B(r)$ as $r\to\ r_0$ means $\lim_{r\to r_0}(A(r)/B(r))=1$ ;
\item $A(r)\propto B(r)$ for all $r\in R$ means that either $A\equiv B\equiv 0$
	on $R$ or $A(r)/B(r)$ is independent of $r\in R$;
\item $A(r)\lesssim B(r)$ [equiv.\ $B(r)\gtrsim A(r)$]
	for all $r\in R$ means that there
	exists a finite constant $c>1$ such that $A(r) \le c B(r)$ for all $r\in R$;
\item $A(r)\asymp B(r)$ for all $r\in R$ means that
	$ B(r)\lesssim A(r)\lesssim B(r)$ for all $r\in R$.
\end{compactitem}

Finally, let us recall that by a ``solution'' $u$ to \eqref{SHE} we mean
a ``mild solution.'' That is:
(i)
$u$ is a predictable random field --- with respect
to the Brownian filtration generated by the cylindrical Brownian motion
defined by $B_t(\phi):=\int_{[0,t]\times\R^3}\phi(y)\,\eta(\d s\,\d y)$,
for all $t\ge0$ and measurable $\phi:\R^3\to\R$ such that
$\<\phi\,,f*\phi\>_{L^2(\R^3)}<\infty$; and (ii)  $u$ solves the stochastic
integral equation,
\begin{equation}\label{mild}
	u(t\,,x) = 1 + \int_{[0,t]\times\R^3}
	p_{t-s}(y-x)\sigma(u(s\,,y))\,\eta(\d s\,\d y),
\end{equation}
where the stochastic integral is understood in the sense of 
Dalang \cite{Dalang} and Walsh \cite{Walsh}.
Finally, it might help to recall also that
\[
	\Cov[B_t(\phi_1)\,,B_s(\phi_2)] = \min(s\,,t)  \<\phi_1\,,f*\phi_2\>_{L^2(\R^3)},
\]
for all $s,t\ge0$ and measurable $\phi_1,\phi_2:\R^3\to\R$ such that
$\<\phi_i\,,f*\phi_i\>_{L^2(\R^3)}<\infty$ for $i=1,2$.

\section{Some classical function theory}

Recall that a function $f:\R^3\to(0\,,\infty)$ is said to be a \emph{correlation
function} if $f$ is locally integrable, with a nonnegative Fourier transform
$\widehat{f}$. The main goal of this section is to establish the following
quantitative variation on a certain form of  Wiener's tauberian theorem.
The following result will be used to show that there are 
many ``bad'' correlation functions on $\R^3$.

Throughout, define $\cB(r)$ to be the centered ball of radius $r$ about the origin;
that is,
\begin{equation}\label{B(r)}
	\cB(r) := \left\{x\in\R^3:\, \|x\|<r \right\}
	\qquad\text{for all $r>0$}.
\end{equation}

\begin{theorem}\label{th:corr}
	For every $\alpha>1$
	there are correlation functions $f:\R^3\to(0\,,\infty)$ such that:
	\begin{compactenum}
		\item $f,\widehat{f}>0$ on $\R^3$;
		\item $f$ is uniformly continuous on $\R^3\setminus\cB(r)$ for every $r>0$;
		\item There exists a nonincreasing function
			$\varphi:\R_+\to\R_+$ such that
			\begin{equation}\label{varphi}
				f(x) = \varphi(\|x\|)\quad
				\text{for all $x\in\R^3$};
			\end{equation}
		\item $f(x) \asymp \|x\|^{-2}\left[\log(1/\|x\|)\right]^{-\alpha}$
			uniformly for all $x\in\cB(1/\e)\setminus\{0\}$; and
		\item $\widehat{f}(x) \asymp {\|x\|^{-1}}\left[\log (\|x\|)\right]^{-\alpha}$
			uniformly for all $x\in \R^3\setminus\cB(1/\e)$.  
		\item $f$ satisfies Dalang Condition \eqref{cond:Dalang}.
	\end{compactenum}
\end{theorem}

We will fix the notation, introduced in
Theorem \ref{th:corr},  for both $f$ and $\varphi$ from now on.

Interestingly enough, we are only aware of one proof. Though most of that proof
can be translated into the language of classical function theory --- specifically,
the theory of Bernstein functions, for example, as described in Schilling,
Song, and Vondra\v{c}ek  \cite{SSV} ---
our proof is decidedly probabilistic at a key point. The reason is that,
thus far, the unimodality result \eqref{weak:unimodal} below only has a
probabilistic derivation, as it depends crucially on the strong Markov property; 
see Khoshnevisan and Xiao \cite[Lemma 4.1]{KX} for details.

From now on, $\alpha>1$ is held fixed.
We follow an idea of Khoshnevisan and Foondun \cite{FK2}, and first define an absolutely-continuous
Borel measure $\nu$ on $(0\,,\infty)$ whose Radon--Nikodym density at $r$
blows up a little bit more slowly than $r^{-3/2}$ as $r\downarrow0$. Specifically, let
\[
	\frac{\d\nu}{\d r}(r) := \begin{cases}
		\dfrac{\left(\log(1/r)\right)^{\alpha}}{r^{3/2}}&\text{if $0< r < \e^{-1}$},\\
		0&\text{otherwise}.
	\end{cases}
\]
Because
\[
	\int_0^\infty (1\wedge r)\,\nu(\d r) =
	\int_0^{\e^{-1}} \frac{  \left(\log(1/r) \right)^{\alpha}}{r^{1/2}}  \d r <\infty,
\]
the structure theory of L\'evy processes --- see Sato \cite[Chapter 4]{Sato} --- tells us that
$\nu$ is the L\'evy measure of a  subordinator
$T :=\{T_t\}_{t\ge0}$ such that $T_0=0$ and
\[
	\E\exp\left( -\lambda T_t\right) = \exp(-t\Phi(\lambda))
	\qquad\text{for all $t,\lambda>0$},
\]
where
\begin{align*}
	\Phi(\lambda) &:= \int_0^\infty\left( 1-\e^{-\lambda t}\right)\,\nu(\d t)\\
	&=\int_0^{\e^{-1}}\frac{(1-\e^{-\lambda t}) \left(\log(1/t)\right)^{\alpha}
		}{t^{3/2}}\,\d t.
\end{align*}
The following lemma describes the asymptotic behavior of $\Phi$
near infinity.

\begin{lemma}\label{lem:Phi}
	$\Phi(\lambda) \sim (4\pi\lambda)^{1/2} \left(\log\lambda \right)^{\alpha}$
	as $\lambda\to\infty$.
\end{lemma}

\begin{proof}
	Thanks to scaling and a simple application of the dominated
	convergence theorem,
	\begin{align*}
		\Phi(\lambda) &= \sqrt\lambda\cdot\int_0^{\lambda\e^{-1}}
			\frac{(1-\e^{- s}) \left(\log(\lambda/s)\right)^{\alpha}}{s^{3/2}
			}\,\d s\\
		&\sim \sqrt{\lambda}\left(\log\lambda\right)^{\alpha} \cdot\int_0^\infty
			\frac{1-\e^{-s}}{s^{3/2}}\,\d s,
	\end{align*}
	as $\lambda\to\infty$.\footnote{Indeed, 
	$0< \log(\lambda/s) \lesssim \log\lambda +  s^{-1/(4\alpha)} + s^{1/(4\alpha)}$
	for all $\lambda>1$ and $s\in(0\,,\lambda/2)$.}
	Now write $1-\e^{-s}=\int_0^s\exp(-r)\,\d r$, plug this into the preceding integral and apply Tonelli's theorem in order to deduce the lemma.
\end{proof}

Next let $U$ denote the 1-potential measure of the subordinator $T$; that is,
for all Borel sets $A\subseteq\R_+$,
\begin{equation}\label{eq:U}
	U(A) := \int_0^\infty
	\P\{T_t\in A\}\e^{-t}\,\d t.
\end{equation}
Evidently, $U$ is a Radon probability measure on $\R_+$. We
refer to this property of $U$ many times in the sequel, and will sometimes
even do so tacitly. 

One can write
\eqref{eq:U} equivalently as follows:
\(
	\int g\,\d U = \E g(T_S),
\)
for all bounded Borel functions $g:\R_+\to\R$, where
$S$ denotes an independent random variable with an exponential,
mean-one distribution.

The following estimates the $U$-measure of a small interval about the origin.

\begin{lemma}\label{lem:U}
	$U[0\,,\varepsilon)\asymp \varepsilon^{1/2}[\log(1/\varepsilon)]^{-\alpha}$
	for all $\varepsilon\in(0\,,\e^{-1})$.
\end{lemma}

\begin{proof}
	The Laplace transform of $U$ can be computed
	easily as follows, thanks to several applications of
	the Fubini--Tonelli theorem: For all $\lambda>0$,
	\begin{equation}\label{LU:Phi}
		(\mathscr{L}U)(\lambda) :=\int_{[0,\infty)}\e^{-\lambda t}\, U(\d t)
		= \E\int_0^\infty \e^{-t-\lambda T_t}\,\d t
		=\int_0^\infty\e^{-t[1+\Phi(\lambda)]}\,\d t
		=\frac{1}{1+\Phi(\lambda)}.
	\end{equation}
	Therefore, Lemma \ref{lem:Phi} ensures that
	\begin{equation}\label{LU}
		(\mathscr{L}U)(\lambda) \sim\frac{1}{(4\pi\lambda)^{1/2} \left(\log\lambda\right)^{\alpha}}
		\qquad\text{as $\lambda\to\infty$}.
	\end{equation}
	Now we apply a standard abelian argument as follows:
	First of all, because
	\[
		(\mathscr{L}U)(\lambda) \ge \int_{[0,1/\lambda)}\e^{-\lambda t}\, U(\d t)
		\ge \e^{-1}U\left[0\,,1/\lambda\right)\qquad\text{for all $\lambda>0$},
	\]
	we can deduce from \eqref{LU} that
	$U[0\,,1/\lambda) \lesssim \lambda^{-1/2}
	(\log\lambda)^{-\alpha}$
	for all $\lambda>\e$.
	In order to obtain the remaining converse bound, let us first recall that
	$U$ is ``$4$-weakly unimodal'' in the sense that
	\begin{equation}\label{weak:unimodal}
		U[x-r\,,x+r)\le 4U[0\,,r)\qquad\text{for every $x\in\R$ and $r>0$;}
	\end{equation}
	see Khoshnevisan
	and Xiao \cite[Lemma 4.1]{KX}. Consequently,
	\[
		(\mathscr{L}U)(\lambda) \le \sum_{n=0}^\infty \e^{-n}
		U\left[n/\lambda\,,(n+1)/\lambda\right)\
		\le 4\sum_{n=0}^\infty\e^{-n}U\left[0\,,1/\lambda\right)
		= \frac{4\e}{\e-1}U\left[0\,,1/\lambda\right).
	\]
	A second appeal to \eqref{LU} completes the proof; to finish
	we simply set $\lambda:=1/\varepsilon$.
\end{proof}

\begin{proof}[Proof of Theorem \ref{th:corr}]
	Let $T$ denote the subordinator that we just constructed
	in Lemma \ref{lem:U}, and
	let $W:=\{W(t)\}_{t\ge0}$ be an independent standard Brownian
	motion in $\R^3$. Then,
	\[
		X_t := W(T_t)\qquad[t\ge0]
	\]
	is an isotropic L\'evy process in $\R^3$. We can see, by
	first conditioning on $T_t$, that the characteristic function of $X$ is given by
	\begin{align*}
		\E\exp(iz\cdot X_t) &= \E\exp\left( -\frac{\|z\|^2}{2} \cdot T_t\right)\\
		&=\exp\left( -t \Phi\left( \|z\|^2/2\right) \right)
			\qquad\text{for all $t\ge0$ and $z\in\R^3$}.
	\end{align*}

	Recall that the heat kernel $p_s(x)$ --- defined in \eqref{p} ---
	is the probability density of $W(s)$ at $x\in\R^3$
	for every $s>0$. Therefore,  for every measurable function $\psi:\R^3\to\R_+$,
	\begin{align*}
		\E\int_0^\infty \psi(X_t)\e^{-t}\,\d t
			& = \int_0^\infty\e^{-t}\,\d t\int_0^\infty\P\{T_t\in\d s\}\
			\int_{\R^3}\d x\ p_s(x)\psi(x)\\
		&=\int_0^\infty U(\d s)\int_{\R^3}\d x\ p_s(x)\psi(x)
			= \int_{\R^3}  \psi(x) f(x)\,\d x,
	\end{align*}
	where
	\begin{equation}\label{eq:def of f}
		f(x) := \int_0^\infty p_s(x)\, U(\d s)\qquad\text{for all $x\in\R^3$}.
	\end{equation}
	This is the function $f$ that was announced in Theorem \ref{th:corr}.
	
	Clearly, $f>0$ on $\R^3$. Also,  Fubini's theorem and \eqref{LU:Phi}
	together imply that the Fourier transform of $f$ is
	\[
		\widehat{f}(z) = \int_0^\infty \e^{-s\|z\|^2/2}\, U(\d s) 
		= (\mathscr{L}U)\left( \|z\|^2/2\right)
		= \frac{1}{1+\Phi\left( \|z\|^2/2\right)}\qquad
		\text{for all $z\in\R^3$.}
	\]
	Among other things, this calculation shows that:
	\begin{compactenum}
		\item[(a)] $0<\widehat{f} \leq 1$; and in particular,
		\item[(b)] $f$ is positive  semi definite.
	\end{compactenum}
	It follows that $f$ is a correlation function, and Part 1 of the theorem is also proved.
	
	Since $U(\R^3)\le1$, and because $(s\,,x)\mapsto p_s(x)$
	is bounded uniformly on $(0\,,\infty)\times[\R^3\setminus\cB(r)]$ for
	every $r>0$, the continuity of $x\mapsto p_s(x)$ and the dominated convergence
	theorem together prove that $f$ is continuous uniformly on $\R^3\setminus\cB(r)$
	for every $r>0$, whence follows part 2 of the theorem.

	Part 3 follows immediately from \eqref{eq:def of f} and the isotropy and monotonicity
	properties of the heat kernel.

	In order to verify part
	4 of the theorem we decompose $f$ as follows:
	\[
		f(x) :=  \underbrace{\int_0^{\e^{-1}} p_s(x)\, U(\d s)}\limits_{:=f_1(x)} +
		\underbrace{\int_{\e^{-1}}^\infty p_s(x)\, U(\d s)}\limits_{:=f_2(x)},
	\]
	Because $p_s(x)\le s^{-3/2}$ for all $x\in\R^3$ and $s>0$,
	\begin{align*}
		\sup_{x\in\R^3}f_2(x) \le \int_{\e^{-1}}^\infty \frac{U(\d s)}{s^{3/2}}
			= \sum_{n=1}^\infty\int_{n\e^{-1}}^{(n+1)\e^{-1}}\frac{U(\d s)}{s^{3/2}}
			&\le \sum_{n=1}^\infty \frac{U[n\e^{-1}\,,(n+1)\e^{-1})}{(n\e^{-1})^{3/2}}\\
		&\le 4\e^{3/2} U[0\,,\e^{-1})\cdot \sum_{n=1}^\infty n^{-3/2}<\infty.		
	\end{align*}
	We have used the weak unimodality [see \eqref{weak:unimodal}] of $U$
	in order to deduce the second line from the first. Therefore, it remains to
	prove that $f_1(x) \asymp \|x\|^{-2}\log(1/\|x\|)^{-\alpha}$
	as long as $\|x\|\le\e^{-1}$.

	Lemma \ref{lem:U} implies the existence of two finite and positive
	constants $a$ and $b$ such that, uniformly for every $\varepsilon\in(0\,,\e^{-1})$,
	\[
		a\varepsilon^{1/2}\log(1/\varepsilon)^{-\alpha}  \le
		U[0\,,\varepsilon) \le b\varepsilon^{1/2}\log(1/\varepsilon)^{-\alpha}.
	\]
	Consequently, as long as we choose a large enough constant
	$K>0$, the following holds
	uniformly when $\|x\|^2<K^{-1}\e^{-1}$:
	\begin{equation}\begin{split}
		f_1(x) &= (2\pi)^{-3/2} \int_0^{\e^{-1}}
			\frac{\exp\left( -\|x\|^2/(2s)\right)}{s^{3/2}}\, U(\d s)\\
		&\ge (2\pi)^{-3/2} \int_{\|x\|^2}^{K\|x\|^2}
			\frac{\exp\left( -\|x\|^2/(2s)\right)}{s^{3/2}}\, U(\d s)
			\\&\ge \frac{ U\left[ 0\,,K\|x\|^2\right) -
			U\left[ 0\,,\|x\|^2\right)}{K^{3/2}2^{3/2}\pi^{3/2}\e^{1/2}\|x\|^3}\\
		&\gtrsim \|x\|^{-2}[\log(1/\|x\|)]^{-\alpha}.
			\label{v1:LB}
	\end{split}\end{equation}
	Since $\varphi$ is monotone, the preceding holds also when $K^{-1}\e^{-1}
	<\|x\|^2\le\e^{-1}$.
	Similarly, we can decompose
	\begin{align*}
		f_1(x) &\le \sum_{\substack{n\in\Z:\\
			\e^{-n}\|x\|^2 \le \e^{-1}}} \int_{\e^{-(n+1)}\|x\|^2}^{\e^{-n}\|x\|^2}
			\frac{\exp(-\|x\|^2/(2s))}{s^{3/2}}\, U(\d s) \\
		&\le \|x\|^{-3}\hskip-20pt\cdot
			\sum_{\substack{n\in\Z:\\
			\e^{-n}\|x\|^2 \le \e^{-1}}}\exp\left(-\frac{\e^n}{2} + \frac{3(n+1)}{2}\right)
			U\left[0\,,\e^{-n}\|x\|^2\right)\\
		&\lesssim \|x\|^{-2}\cdot
			\sum_{\substack{n\in\Z:\\
			\e^{-n}\|x\|^2 \le \e^{-1}}}\exp\left(-\frac{\e^n}{2} + n\right)
			\left[\log\left(\e^n/\|x\|^2\right)\right]^{-\alpha}.
	\end{align*}
	This readily yields the complementary bound to \eqref{v1:LB}
	and completes the proof of part 4. 
	
	Part 5 was proved in Foondun and Khoshnevisan \cite{FK2};
	see the argument that led to Theorem 3.14 therein ({\it ibid.}). 
	
	In order to complete the proof we verify part 6 of the theorem. 
	 Let $\{\bar{R}_\lambda\}_{\lambda>0}$ denote the resolvent
	of the heat semigroup on $\R^3$, run at twice the standard speed; that is,
	\begin{equation}\label{R}
		(\bar{R}_\lambda g)(x) :=
		\int_0^\infty (p_{2s}*g)(x)\e^{-s\lambda}\,\d s,
	\end{equation}
	for all functions $g:\R^3\to\R$ for which the preceding Lebesgue integral is defined.
	The theory of Foondun and Khoshnevisan \cite[Theorem 1.2]{FK2} implies that 
	Dalang's Condition \eqref{cond:Dalang} is equivalent to 
	the condition that $(\bar{R}_1f)(0)<\infty$. Therefore,
	it remains to prove that $(\bar{R}_1f)(0)$ is finite.
	 This is a well-known calculation about the Newtonian potential in dimension three.
	 The computations will be carried out here for the sake of completeness.
	 
	One can integrate in spherical coordinates as follows:
	\begin{align*}
		0\le (\bar{R}_1f)(0) &\le \int_0^\infty (p_{2s}*f)(0)\,\d s
		= \int_{\R^3} f(x)\,\d x\int_0^\infty p_{2s}(x)\,\d s
		\propto \int_{\R^3} \frac{f(x)}{\|x\|}\,\d x.
	\end{align*}
	Therefore, the lemma follows once one proves that $\int_{\R^3}\|x\|^{-1}f(x)\,\d x<\infty$;
	that is, once we prove that $f$ has finite Newtonian potential. 
	
	Note that
	\[
		\int_{\R^3} \frac{p_s(x)}{\|x\|}\,\d x \propto s^{-1/2},
	\]
	uniformly for all $s>0$.  It follows from the definition \eqref{eq:def of f} 
	of $f$ that the Newtonian potential of $f$ can be written as 
	\[
		\int_{\R^3}\frac{f(x)}{\|x\|}\,\d x\propto \int_0^\infty \frac{U(\d s)}{\sqrt s}
		\le \int_0^1\frac{U(\d s)}{\sqrt s} + U[1\,,\infty)
		\le \int_0^1\frac{U(\d s)}{\sqrt s} +1.
	\]
	It remains to prove that $\int_0^1 s^{-1/2}\, U(\d s)$ is finite;
	this  endeavor will complete the proof since $U$ is a probability measure. To see 
	that $\int_0^1 s^{-1/2}\, U(\d s)$ is finite, one 
	simply integrates by parts,
	\[
		\int_0^1 \frac{U(\d s)}{\sqrt s} =
		\frac12\int_0^1\frac{U[0\,,r]}{r^{3/2}}\,\d r + U[0\,,1]
		\le  \frac12\int_0^1\frac{U[0\,,r]}{r^{3/2}}\,\d r+1,
	\]
	and apply Lemma \ref{lem:U} together with the fact that $\alpha>1$.
	This completes the proof.
\end{proof}

\section{Preliminary estimates}

Recall that we are studying \eqref{SHE} for a Lipschitz-continuous, non-random function
$\sigma:\R\to\R$, subject to $u(0)\equiv1$.

From now on, we restrict attention to a noise model for $\eta$ that
corresponds to a spatial correlation function $f:\R^3\to\R_+$ that satisfies 
properties 1--5 of Theorem \ref{th:corr}; the choice of $f$ is otherwise arbitrary.

\subsection{Existence, uniqueness, and moments}

The following result follows from Theorems \ref{th:Dalang} and
\ref{th:corr} (part 6).

\begin{lemma}\label{lem:exist}
	The stochastic partial differential equation \eqref{SHE}, subject to $u(0)\equiv1$
	admits a predictable random field solution $u$. Moreover, $u$
	is unique, up to a modification, subject to the condition that for all $T\in(0\,,\infty)$
	and $k\in[2\,,\infty)$,
	\[
		\sup_{t\in[0,T]}\sup_{x\in\R^3} \E\left( |u(t\,,x)|^k\right)<\infty.
	\]
\end{lemma}

\begin{remark}
	Recall $\varphi$ from \eqref{varphi} and note that
	$\int_{\R^3} \|x\|^{-1}f(x)\,\d x\propto\int_0^\infty r\varphi(r)\,\d r$.
	Thus, Lemma \ref{lem:exist} follows from the fact that $\int_0^\infty r\varphi(r)\,\d r<\infty$;
	see the proof of Lemma \ref{lem:exist}.   This sort of integrability condition for $r\mapsto r\varphi(r)$
	arose earlier in the context of hyperbolic SPDEs; see Dalang and Frangos \cite{DF}.
\end{remark}

Next we produce moment estimates for the solution to \eqref{SHE}, all
the time remembering that the spatial
correlation function $f$ of $\eta$ satisfies the properties mentioned in
Theorem \ref{th:corr}, and $\alpha>1$ is the underlying parameter that was
used in the course of the construction of $f$.

\begin{theorem}\label{th:LyapunovExp}
	Let $u$ denote the solution to \eqref{SHE}, and
	recall that $\sigma$ is Lipschitz continuous and non random. 
	Then, there exists a finite constant $A>0$ such that
	\begin{equation}\label{LE:UB}
		 \E\left(|u(t\,,x)|^k\right) \le
		A^k\exp\left[A\exp\left( A k^{1/(\alpha-1)}\right)\cdot t\right], 
	\end{equation}
	uniformly in $x\in \R^3$ and $k\in[2\,,\infty)$ and $t >0$. 
	For a complementary bound, suppose that $\sigma(z)=z$ for all $z\in\R$.
	Then, in that case, there exists a finite constant $A_1>0$ such that 	
	\begin{equation}\label{LE:LB}
		\E\left(|u(t\,,x)|^k\right) \geq
		A_1^k\exp\left[A_1\exp\left( A_1 k^{1/\alpha}\right)\cdot t\right], 
	\end{equation}
	uniformly for all $x\in\R^3$ and all integers $k\ge2$ and $t>0$.
\end{theorem}

\begin{remark}\label{rem:LE}
	Inequality \eqref{LE:LB} is included here mainly because it shows
	that, for the spatial correlation function $f$ of the type studied here,
	the solution to \eqref{SHE} is ``extremely intermittent.'' One way to say this
	is as follows: Consider the [lower] \emph{moment Lyapunov exponents},
	\[
		\gamma(k) := \liminf_{t\to\infty} t^{-1}\log\E\left(
		|u(t\,,0)|^k\right)
		\qquad\text{for all $k\in[2\,,\infty)$}.
	\]
	Then, \eqref{LE:LB} proves that
	\(
		\liminf_{k\to\infty} k^{-1/\alpha} \log\gamma(k) >0.
	\)
	In other words, the Lyapunov moments exponents grow extremely
	rapidly with the moment numbers. For usual choices of the spatial correlation
	function $f$,
	$\log\gamma(k)$ grows as $\log k$, whereas it grows as $k^{1/\alpha}$
	here. This sort of extreme intermittency provides a certain amount of
	evidence toward the truth of {Conjecture \ref{Conj2}}, though it certainly does not
	prove {Conjecture \ref{Conj2}}.
\end{remark}

In order to prove the upper bound \eqref{LE:UB} 
we will use a general result of Foondun and Khoshnevisan \cite[Theorem 1.3]{FK2}.
For the lower bound \eqref{LE:LB} we first use a Feynman--Kac type moment formula to represent 
the solution, and then reduce the problem to a small-ball estimate for three-dimensional 
Brownian motion. 

The proof of the upper bound requires two technical lemmas
which we develop next.

\begin{lemma}\label{lem:f}
	\(
		(p_{2t}*f)(0) \asymp t^{-1} [\log (1/t)]^{-\alpha}
	\)
	uniformly for all $t\in(0\,,\e^{-1})$.
\end{lemma}

\begin{proof}
	We find it more convenient to work with $p_t$ rather than $p_{2t}$;
	a change of variables $[2t\rightarrow t]$ will adjust the constants for
	correct later use.

	We integrate in spherical coordinates to see that,  for all $t>0$,
	\[
		(p_t*f)(0)\propto t^{-3/2}\cdot\int_0^\infty
		r^2 \e^{-r^2/(2t)} \varphi(r)\,\d r 
		\propto
		\int_0^\infty s^2 \e^{-s^2/2} \varphi\left(s\sqrt{t}\right)\d s
		\propto  \cT_1 + \cT_2,
	\]
	where
	\[
		\cT_1 := \int_0^{1/\sqrt{2t}} s^2 \e^{-s^2/2} \varphi\left(s\sqrt{t}\right)\d s,\qquad
		\cT_2 := \int_{1/\sqrt{2t}}^\infty s^2 \e^{-s^2/2} \varphi\left(s\sqrt{t}\right)\d s,
	\]
	both are functions of the time variable $t$ which we suppress.

	The second quantity $\cT_2$ is bounded uniformly in $t$. In fact, the monotonicity
	of $\varphi$ --- see Theorem \ref{th:corr} --- yields
	\[
		\cT_2 \le  \varphi\left(1/\sqrt 2\right)\int_0^\infty s^2 \e^{-s^2/2}\,\d s<\infty.
	\]
	Therefore, it remains to prove that $\cT_1\asymp t^{-1}|\log t|^{-\alpha}$ for all $t\in(0\,,t_0)$,
	where $t_0>0$ is a sufficiently-small constant. 
	
	Choose and fix a constant $K>2\e$.
	Thanks to part  4 of Theorem \ref{th:corr}, we can write
	\[
		\cT_1 \asymp t^{-1}\int_0^{1/\sqrt{Kt}}
		\left| \log\left(\frac{1}{s\sqrt{t}}\right)\right|^{-\alpha}
		\e^{-s^2/2}\,\d s
		:= \frac{\cT_{1,1}+\cT_{1,2}}{t},
	\]
	where
	\begin{align*}
		\cT_{1,1} &:= \int_0^{\sqrt{K\log(1/t)}}
		\left| \log\left(\frac{1}{s\sqrt{t}}\right)\right|^{-\alpha}
		\e^{-s^2/2}\,\d s\,\\
		\cT_{1,2} &:= \int_{\sqrt{K\log(1/t)}}^{1/\sqrt{Kt}} \left| 
		\log\left(\frac{1}{s\sqrt{t}}\right)\right|^{-\alpha} \e^{-s^2/2}\,\d s.
	\end{align*}
	If $0<s<\sqrt{K\log (1/t)}$, then
	$|\log ( s\sqrt  t)| \gtrsim \log(1/t)$.
	Therefore,
	\[
		\cT_{1,1} \lesssim[\log(1/t)]^{-\alpha}\cdot
		\int_0^\infty \e^{-s^2/2}\,\d s
		\lesssim [\log(1/t)]^{-\alpha}.
	\]
	Similarly,
	\[
		\cT_{1,2} \le \left|\log\left(1/\sqrt K\right)\right|^{-\alpha}
		\cdot
		\int_{\sqrt{2\log(1/t)}}^\infty \e^{-s^2/2}\,\d s
		=o\left([\log(1/t)]^{-\alpha}\right)\qquad\text{as $t\downarrow0$}.
	\]
	In this way we have proved that $\cT_1\lesssim t^{-1}[\log (1/t)]^{-\alpha}$, whence
	also 
	\[
		(p_t*f)(0)\lesssim t^{-1}[\log (1/t)]^{-\alpha},
	\]
	for small values of $t>0$. 
	
	The  other bound is even simpler to establish since
	\[
		\cT_{1,1}  + \cT_{1,2} \ge \int_1^2\left| \log\left(
		\frac{1}{s\sqrt{t}}\right)\right|^{-\alpha}
		\e^{-s^2/2}\,\d s
		\ge \left| \log\left(
		\frac{1}{\sqrt{t}}\right)\right|^{-\alpha}\cdot\int_1^2\e^{-s^2/2}\,\d s,
	\]
	for all sufficiently-small values of $t$.
\end{proof}

Let $R:=\{R_\lambda\}_{\lambda>0}$ denote
the resolvent of the Laplace operator $\frac12\Delta$;
compare with \eqref{R}.
We can write $R$ in terms of the heat kernel of Brownian motion
as
\[
	(R_\lambda h)(x) := \int_0^\infty (p_{s}*h)(x)
	\e^{-\lambda s}\,\d s,
\]
for all $x\in\R^3$, $\lambda>0$, and Borel functions $h:\R^3\to\R_+$.

\begin{lemma}\label{lem:R}
	$(R_\lambda f)(0)  \asymp (\log\lambda)^{-\alpha+1}$
	for all $\lambda\ge \e$.
\end{lemma}

\begin{proof}
	First, let us observe that
	$(R_\lambda f)(0)<\infty$ for all $\lambda>0$ for the same
	sort of reason that showed that $(R_1f)(0)<\infty$; see
	the proof of Lemma \ref{lem:exist}.

	Next we write $(R_{\lambda}f) (0) := \cT_1+\cT_2$, where
	$\cT_i=\cT_i(\lambda)$ [$i=1,2$] are defined as follows:
	 \[
		\cT_1 := \int_0^{1/\lambda}  (p_t* f)(0) \e^{-\lambda t}\,\d t
		\quad\text{and}\quad
		\quad \cT_2:= \int_{1/\lambda}^{\infty}  (p_t* f)(0) \e^{-\lambda t}\,\d t.
          \]
	We estimate $\cT_1$ and $\cT_2$ separately.
	
	A change of variable shows that
	  \[
		\cT_1 = \lambda^{-1}\int_0^1  \left( p_{u/\lambda}*f \right) (0)\e^{-u}\,\d u
		\asymp \lambda^{-1} \int_0^1 \left(p_{u/\lambda}*f\right)(0)\,\d u.
	  \]
	  Therefore, Lemma \ref{lem:f} ensures that
	  \[
		\cT_1 \asymp \int_0^1\frac{\d u}{u |\log(\lambda/u)|^\alpha}
		\asymp \left( \log  \lambda \right)^{-\alpha+1}.
	  \]
	  Because $f>0$ (Theorem \ref{th:corr}), 
	  it remains to prove that $\cT_2=o(\cT_1)$ as $\lambda\to\infty$.
	  
	  By the semigroup property of the heat kernel,
	  \[
	  	(p_{t+s}*f)(0) = (p_s*p_t*f)(0) 
		\le \sup_{x\in\R^3}(p_t*f)(x).
	  \]
	 Since $p_t*f$ is a continuous, positive semi-definite function, it is maximized
	 at the origin. Therefore, we can deduce from the preceding display that $t\mapsto(p_t*f)(0)$ 
	 is non increasing. In particular,
	 \[
	 	\cT_2 \le \left( p_{1/\lambda}*f\right)(0)\cdot
		\int_{1/\lambda}^\infty \e^{-\lambda t}\,\d t 
		\lesssim\lambda^{-1}\left( p_{1/\lambda}*f\right)(0)
		\asymp ( \log\lambda)^{-\alpha},
	 \]
	 thanks to Lemma \ref{lem:f}. This and the estimate for $\cT_1$
	 together imply that $\cT_1=o(\cT_2)$ as $\lambda\to\infty$, which completes the proof.
\end{proof}

\begin{proof}[Proof of Theorem \ref{th:LyapunovExp}]
	First we prove the claimed upper bound on the moments of
	$u(t\,,x)$.

	According to Lemma \ref{lem:R},
	\[
		Q(k\,,\lambda) := \frac{k}{\lambda} + 2\sqrt{k\left(
		R_{2\lambda/k}f\right)(0)}
		\lesssim \frac{k}{\lambda} +
		\sqrt{\frac{k}{|\log(\lambda/k)|^{\alpha-1}}},
	\]
	uniformly for all $\lambda\ge e k/2$ and $k\ge 2$. If, in addition,
	\begin{equation}\label{lllk}
		\log\lambda > Ck^{1/(\alpha-1)},
	\end{equation}
	for a sufficiently large $C>0$, then the preceding simplifies to the following inequality:
	\[
		Q(k\,,\lambda) \lesssim
 \frac{k^{1/2}}{(\log\lambda)^{(\alpha-1)/2}}   <1,
	\]
	In particular, \eqref{lllk} tells us that there exists a positive and finite constant
	$A$ such that
	\[
		\inf\left\{\lambda>0:\ Q(k\,,\lambda)<1\right\}
		\le  A\exp\left( Ak^{1/(\alpha-1)}\right)  \qquad\text{for all
		$k\ge 2$.}
	\]
	Theorem 1.3 of Foondun and Khoshnevisan \cite{FK2} now shows that
	\[
		\limsup_{t\to\infty}t^{-1}\log \sup_{x\in\R^3} \E\left(|u(t\,,x)|^k\right)
		\le A\exp\left(Ak^{1/(\alpha-1)}\right),
	\]
	for all $k\ge 2$. This proves an asymptotic, large-$t$, version of the
	stated upper bound \eqref{LE:LB} of the theorem. The asserted
	fixed-$t$ result holds because of the proof
	of Theorem 1.3 of Foondun and Khoshnevisan ({\it ibid.}); consult
	Lemmas 5.4 and 5.5 of that reference for details.

	To prove the lower bound \eqref{LE:LB} let us recall the following Feynman--Kac
	formula for the moments of the parabolic Anderson model; see
	Hu and Nualart \cite{HN} and Conus \cite{Conus}:
	\begin{align*}
		\E\left(|u(t\,,x)|^k\right) &= \E\left[\exp\left(
			\mathop{\sum\sum}\limits_{
			1\le i< j\le k}\int_0^tf\left(w^{(i)}_s-w^{(j)}_s\right)\d s\right)\right]\\
		&=\E\left[\exp\left(\mathop{\sum\sum}\limits_{
			1\le i< j\le k}\int_0^t  \varphi\left(\|w^{(i)}_s-w^{(j)}_s\|\right)\d s
			\right)\right],
	\end{align*}
	where $w^{(1)},\ldots,w^{(k)}$ are independent, standard Brownian motions
	on $\R^3$. Let $\bm{E}_\eta:=\bm{E}_{\eta,k,t}$ denote the
	event that $\|w^{(i)}_s\|\le \eta$ for all $1\le i\le k$ and $s\in[0\,,t]$. Then clearly,
	\[
		\E\left(|u(t\,,x)|^k\right) \ge\E\left[\exp\left(
		\mathop{\sum\sum}\limits_{
		1\le i< j\le k}\int_0^t  \varphi\left(\|w^{(i)}_s-w^{(j)}_s\|\right)\d s\right);\,
		\bm{E}_\eta\right].
	\]
	Thanks to \eqref{varphi}, if $0<\eta\le\exp(-\e)$, then we can find a positive
	constant $L$ such that, uniformly for all
	$s\in[0\,,t]$ and $1\le i\le k$,
	$ \varphi (\|w^{(i)}_s-w^{(j)}_s\|)\ge L\eta^{-2}|\log\eta|^{-\alpha}$
	almost surely on $\bm{E}_\eta$. Thus, we see that
	\begin{equation}\label{LB}
		\log\E\left(|u(t\,,x)|^k\right) \ge\sup_{\eta\in(0,\exp(-\e))}\left[
		\log\P(\bm{E}_\eta) + \frac{Lk(k-1)t}{2\eta^2(\log(1/\eta))^{\alpha}}
		\right].
	\end{equation}
	It is well known, and easy to see directly, that there exists a universal
	positive constant $c$ such that
	\[
		\P(\bm{E}_\eta) = \left[\P\left\{ \sup_{0\le s\le t}\|w^{(1)}_s\|
		\le \eta\right\}\right]^k
		\ge \exp\left(- ckt/\eta^2\right),
	\]
	uniformly for all $t>0$ and $\eta\in(0\,,\e^{-1})$. [The preceding
	probability can in fact be computed explicitly; see Ciesielski and Taylor \cite[Theorem 2]{CT}.]
	We plug this inequality into \eqref{LB}.
	A line or two of further computations conclude the proof. 
\end{proof}

The proofs of Theorems \ref{th:conditional} and \ref{th:local} will
rely on the following variation on the theme of Theorem \ref{th:LyapunovExp}.

\begin{proposition}\label{pr:LyapunovExp}
	Suppose, in addition, that $\sigma$ is bounded. Then,
	\[
		\sup_{x\in\R^3}\E\left(|u(t\,,x)|^k\right) \lesssim (1 + kt)^{k/2},
	\]
	uniformly for all $(t\,,x\,,k)\in\R_+\times\R^3\times[2\,,\infty)$.
\end{proposition}

\begin{proof}
	In accord with \eqref{mild} and suitable form of the Burkholder--Davis--Gundy inequality
	\cite{Khoshnevisan},
	\begin{align*}
		&\|u(t\,,x)\|_k^2  \lesssim 1 + \left\| 
			\int_{[0,t]\times\R^3}
			p_{t-s}(y-x)\sigma(u(s\,,y))\,\eta(\d s\,\d y)\right\|_k^2\\
		&\le 1+4k\int_0^t\d s\int_{\R^3}\d y\int_{\R^3}\d y'\
			p_{t-s}(y-x) p_{t-s}(y'-x){\left\| \sigma(u(s\,,y))\cdot
			\sigma(u(s\,,y'))\right\|_{k/2}}f(y-y')\\
		&\lesssim 1 + k\int_0^t\d s\int_{\R^3}\d y\int_{\R^3}\d y'\
			p_{t-s}(y-x) p_{t-s}(y'-x) f(y-y'),
	\end{align*}
	uniformly for all $(t\,,x\,,k)\in\R_+\times\R^3\times[2\,,\infty)$.
	Now, the boundedness of the function $\sigma$ simplifies the
	preceding as follows:
	\begin{align*}
		\|u(t\,,x)\|_k^2 &\lesssim 1+k
			\int_0^t\d s\int_{\R^3}\d y\int_{\R^3}\d y'\ p_s(y)p_s(y')f(y-y')\\
		&=1+k\int_0^t\d s\int_{\R^3}\d y\ p_s(y)(p_s*f)(y)\\
		&=1+k\int_0^t(p_{2s}*f)(0)\,\d s.
	\end{align*}
	Since $p_{2s}$ and $f$ are both positive semi-definite, so is their convolution.
	Moreover, $p_{2s}*f$ is manifestly continuous. Therefore, by elementary properties
	of continuous, positive definite functions,
	$p_{2s}*f$ is maximized at the origin. Therefore, Lemma \ref{lem:f} implies that
	\[
		(p_{2s}*f)(w) \le (p_{2s}*f)(0)
		\lesssim \frac{1}{s[\log(1/s)]^\alpha}\qquad\text{for all
		$s\in(0\,,\e^{-1}]$ and $w\in\R^3$}.
	\]
	Moreover, the semigroup property of the heat kernel implies that
	$p_{2s}=p_{{2/\e}}*p_{2(s-(1/\e))}$ for all $s>\e^{-1}$,
	whence
	\[
		(p_{2s}*f)(0) \lesssim \left[ p_{2(s-(1/\e))} *(p_{2(1/\e)}*f)\right](0)
		\le\sup_{w\in\R^3}(p_{2(1/\e)}*f)(w)
		\lesssim 1,
	\]
	uniformly for all $s\ge\e^{-1}$.
	Consequently,
	\[
		\|u(t\,,x)\|_k^2 \lesssim 1 + k\int_0^t
		\max\left\{ 1~,~\frac{1}{s[\log(1/s)]^\alpha}\right\}\d s
		\asymp 1+kt,
	\]
	uniformly for all $(t\,,x\,,k)\in\R_+\times\R^3\times[2\,,\infty)$.
\end{proof}

\begin{remark}\label{rem:Var[Z]}
	In order to highlight the efficacy of Proposition \ref{pr:LyapunovExp}, let us
	consider the case that $\sigma$ is a constant function; say,
	$\sigma\equiv1$. In that case, $u(t\,,x)$ is a mean-one Gaussian
	random variable with variance,
	\[
		\Var[u(t\,,x)] =  \int_0^t\d s\int_{\R^3}\d y
		\int_{\R^3}\d y'\ p_{t-s}(y-x)p_{t-s}(y'-x)f(y-y')
		\asymp 1+t,
	\]
	by the same sort of computation as the one used in the course of
	the proof of Proposition \ref{pr:LyapunovExp}. Special properties
	of mean-zero Gaussian distributions then imply that when $\sigma\equiv1$,
	\[
		\E\left(|u(t\,,x)-1|^k\right) \asymp k^{k/2}\left\{ \Var[u(t\,,x)]\right\}^{k/2}
		\asymp k^{k/2}(1+t)^{k/2},
	\]
	uniformly for all $(t\,,x\,,k)\in\R_+\times\R^3\times[2\,,\infty)$.
	One can conclude readily from this inequality that the statement of
	Proposition \ref{pr:LyapunovExp} is, in its essence, unimprovable.
\end{remark}

\subsection{Moment bounds for the spatial and temporal increments}
In this subsection we give estimates for the quantity
\[
	\E\left(|u(t\,,x) - u(t'\,,x')|^k\right),
\]
when $t\approx t'$ and $x\approx x'$.
These estimates will be used in an ensuing ``local linearization argument''
that will be highlighted in Proposition \ref{pr:localization}. 
Throughout this paper, we set $\log_+\theta:=\log(\theta\vee\e)$ for all $\theta\in\R$.

\begin{proposition}\label{pr:u(x)-u(x')}
	 Assume that $\sigma$ is bounded. Then for all $T\in(0\,,\infty)$ 
	 there exists a finite constant $A$ depending on $T$ such that
	\[
		\sup_{t\in(0,T)}
		\E\left(|u(t\,,x) - u(t\,,x')|^k\right) \le 
		\frac{{(Ak)^{k/2}}}{%
		\left[\log_+(1/\|x-x'\|)\right]^{(\alpha-1) k/2}},
	\]
	uniformly for all distinct $x,x'\in\R^3$, and all real numbers $k\ge2$.
\end{proposition}

\begin{proof}
	Choose and fix $t>0$ and $x,x'\in\R^3$.
	According to \eqref{mild} and a suitable application of
	the BDG inequality (see \cite{Khoshnevisan} for details), for every real number $k\ge 2$,
	\begin{align*}
		&\E\left(\left| u(t\,,x) - u(t\,,x')\right|^k\right)\\ 
		&= \E\left(\left| \int_{(0,t)\times\R^3}
			(p_{t-s}(y-x)-p_{t-s} (y-x') ) \sigma(u(s\,,y))\,\eta(\d s\,\d y)\right|^k\right)\\
		&\le (4k)^{k/2}\left[\int_0^t\d s\int_{\R^3}\d y\int_{\R^3}\d y'\
			\left|\mathcal{P}_{t-s}({y})\mathcal{P}_{t-s}({y'}) \right|
			f(y-y')\mathcal{E}(s\,,y){\mathcal{E}(s\,,y')} \right]^{k/2},
	\end{align*}
	where
	\[
		\mathcal{P}_r(a) := p_r(a-x)-p_r(a-x')\qquad\text{for all $r>0$ and $a\in\R^3$},
	\]
	and 
	\[
		\mathcal{E}(s\,,y) := \left\{ \E\left( | \sigma(u(s\,,y))|^k\right)\right\}^{1/k}\,.
	\]
	Since $\sigma$ is {bounded},
	there exists a finite constant $B>1$ such that uniformly for
	 all $t>0$, $k\ge 2$, $y\in\R^3$, and $s\in[0\,,t]$,
	\begin{align}
		\E\left(\left| u(t\,,x) - u(t\,,x')\right|^k\right)
			&\le B^k k^{k/2}
			\left[\int_0^t\d s\int_{\R^3}\d y\int_{\R^3}\d y'\
			\left| \mathcal{P}_s(y)\mathcal{P}_s(y')\right|
			f(y-y')\right]^{k/2}
			\label{E(u-u)}\\\notag
		&= B^k k^{k/2}
			\left[\int_0^t\d s\int_{\R^3}\d y\int_{\R^3}\d y'\
			\left| \mathcal{P}_s(y)\mathcal{P}_s(y')\right|
			\varphi(\|y-y'\|)\right]^{k/2};
	\end{align}   
	see \eqref{varphi} for the definition of $\varphi$.
	At first one might expect that the absolute values in the integral
	introduce additional logarithmic factors which can damage our
	estimates since the left-hand side is quite large already [remember
	that we are trying to prove that the left-hand side is at most
	a negative power of the iterated logarithm of $\|x-x'\|$]. Remarkably,
	the introduction of the absolute values turns out to be harmless.
	In order to prove this we will use the following elementary inequality: Uniformly for all
	$z\in\R^3$ and $s>0$,
	\begin{equation}\label{L1}
		\int_{\R^3}  |p_s(y-z)-p_s(y)|\,\d y\lesssim
		\frac{\|z\|}{\sqrt s}\wedge 1.
	\end{equation}
	For a detailed proof see Lemma 6.4 in \cite{CJKS}. Now we analyze \eqref{E(u-u)}.

	Throughout the remainder of this calculation, let us define
	\[
		z := x-x'.
	\]
	Theorem \ref{th:corr} and
	two back-to-back applications of \eqref{L1} together show that
	\begin{equation}\label{CASE1}\begin{split}
		\int_0^t\d s\iint_{\substack{y,y\in\R^3:\\|y-y'|>\|z\|}}
			\d y&\,\d y'\
			\left| \mathcal{P}_s(y)\mathcal{P}_s(y')\right|
			\varphi(\|y-y'\|)\lesssim \varphi (\|z\|)
			\int_0^t\left( \frac{\|z\|}{\sqrt s}\wedge 1\right)^2\,\d s\\
		&\lesssim \|z\|^{-2}[\log_+(1/\|z\|)]^{-\alpha}\cdot
			\int_0^t\left( \frac{\|z\|}{\sqrt s}\wedge 1\right)^2\,\d s\\
		&\lesssim [\log_+(1/\|z\|)]^{1-\alpha}.
	\end{split}\end{equation}
	Next we estimate the same integral as above,
	but with its region of integration
	replaced by $\{y,y'\in\R^3:\, \|y-y'\|\le\|z\|\}$. [It might help
	to consult \eqref{E(u-u)} to see why.]
	With this aim
	in mind, define for all $y\in\R^3$ and for every integer $n\ge 0$,
	\[
		\mathcal{A}_n(y) := \left\{ y'\in\R^3:\
		\|y-y'\|\le 2^{-n}\|z\|\right\}.
	\]
	By the monotonicity properties of $\varphi$ [Theorem \ref{th:corr}],
	\begin{equation}\label{p<H}
		\int_0^t\d s\hskip-.3in\mathop{\iint}\limits_{\substack{y,y'\in\R^3:\\
		y'\in\mathcal{A}_n(y)\setminus\mathcal{A}_{n+1}(y)}}
		\hskip-.7cm\d y\,\d y'\
		\left| \mathcal{P}_s(y)\mathcal{P}_s(y')\right|\varphi(\|y-y'\|)\\
		\lesssim \varphi\left(2^{-n-1}\|z\|\right)
		\int_0^t H_n(s)\,\d s,
	\end{equation}
	all the time noting that the implied constant also
	does not depend on $(x\,,x')$---whence also on $z$---and
	\[
		H_n(s) := \int_{\R^3}\d y
		\int_{\mathcal{A}_n(y)\setminus\mathcal{A}_{n+1}(y)}\d y'\
		\left| p_s(y-z) - p_s(y)\right|\cdot\left| p_s(y'-z)-p_s(y')\right|.
	\]
	The elementary properties of the heat kernel $p$ and the inequality
	\eqref{L1} together allow us to write
	\begin{align*}
		\int_{\mathcal{A}_n(y)\setminus\mathcal{A}_{n+1}(y)}
			\left| p_s(y'-z)-p_s(y')\right|\,\d y'
			&\le \int_{\mathcal{A}_n(y)\setminus\mathcal{A}_{n+1}(y)}
			\left| p_s(y'-z)+p_s(y')\right|\,\d y'  \\
		&\lesssim \left(\frac{|\mathcal{A}_n(y)\setminus
			\mathcal{A}_{n+1}(y)|}{s^{3/2}}\wedge 1\right)\\
		&\lesssim\left(\frac{2^{-n}\|z\|}{s^{1/2}}\wedge 1\right)^3,
	\end{align*}
	where the implied constant does not depend on $(n\,,s\,,z)$. Therefore,
	\[
		H_n(s) \lesssim
		\left(\frac{\|z\|}{\sqrt s}\wedge 1\right)
		\cdot\left(\frac{2^{-n}\|z\|}{\sqrt{s}}\wedge 1\right)^3,
	\]
	where the implied constant does not depend on $(n\,,s\,,z)$.
	In particular,
	\begin{align*}
		\int_0^tH_n(s)\,\d s &\lesssim\int_0^t
			\left(\frac{\|z\|}{\sqrt s}\wedge 1\right)
			\cdot\left(\frac{2^{-n}\|z\|}{\sqrt{s}}\wedge 1\right)^3\,\d s\\
		&\lesssim \frac{\|z\|^2}{4^n} + \frac{\|z\|^3}{8^n}
			\int_{4^{-n}\|z\|^2}^{\|z\|^2} s^{-3/2}\,\d s + \frac{\|z\|^4}{8^n}
			\int_{\|z\|^2}^\infty s^{-2}\,\d s\\
		&\lesssim\frac{\|z\|^2}{4^n},
	\end{align*}
	where the implied constant does not depend on $(n\,,t\,,z)$.
	In light of the preceding bound and \eqref{p<H}, we find that
	\begin{align*}
		&\int_0^t\d s\iint_{\substack{y,y\in\R^3:\\|y-y'|\leq\|z\|}}
			\d y\ \d y'\
			\left| \mathcal{P}_s(y)\mathcal{P}_s(y')\right|
			\varphi(\|y-y'\|) \\
		&\hskip2in=\sum_{n=0}^\infty
			\int_0^t\d s\hskip-.3in\iint\limits_{\substack{y,y'\in\R^3:\\
			y'\in\mathcal{A}_n(y)\setminus\mathcal{A}_{n+1}(y)}}
			\hskip-.7cm\d y\,\d y'\
			\left| \mathcal{P}_s(y)\mathcal{P}_s(y')\right|\varphi(\|y-y'\|)\\
		&\hskip2in\lesssim\sum_{n=0}^\infty \frac{\|z\|^2}{4^n}
			\varphi\left( 2^{-n-1}\|z\|\right),
	\end{align*}
	where the implied constants do not depend on $(t\,,z)$. Therefore,
	Theorem \ref{th:corr} ensures that
	\begin{align}\notag
		\int_0^t\d s\iint_{\substack{y,y\in\R^3:\\|y-y'|\leq\|z\|}}
			\d y \ \d y'\
			\left| \mathcal{P}_s(y)\mathcal{P}_s(y')\right|
			\varphi(\|y-y'\|) &\lesssim\sum_{n=0}^\infty \frac{1}{
			(\log_+(2^{n+1}/\|z\|))^{\alpha}}\\\notag
		&\lesssim\|z\|\cdot
			\sum_{n=0}^\infty 2^{-n} \int_{2^n/\|z\|}^{2^{n+1}/\|z\|}
			\frac{\d r}{(\log_+ r)^{\alpha}}\\\notag
		&\lesssim\int_{1/\|z\|}^\infty\frac{\d r}{r (\log_+ r)^{\alpha}
			}\\
		&\propto \left[\log_+(1/\|z\|)\right]^{-\alpha+1},
		\label{CASE2}
	\end{align}
	and, as before, the implied constants do not depend on $(t\,,z)$. In
	light of \eqref{E(u-u)}, \eqref{CASE1}, and \eqref{CASE2}, we can deduce the existence of
	a finite constant $B$ such that,
	uniformly for
	 all $0<t <T$ and $k\ge 2$,
	\[
		\E\left(\left| u(t\,,x) - u(t\,,x')\right|^k\right)\\
		\le \frac{{B^k k^{k/2}}}{%
		\left[\log_+\left( 1/\|x-x'\|\right)\right]^{(\alpha-1)k/2}},
	\]
	where $B$ is also independent of $(x\,,x')$.
	We can replace $A$ by a possibly larger constant in order to complete
	the proof of the result.
\end{proof}




Next, let us consider bounds on the temporal increments of the solution to
\eqref{SHE}. The main result of the section is recorded as the following proposition. 

\begin{proposition}\label{pr:u(t)-u(t')}
	Assume that $\sigma$ is bounded. Then for all $T \in (0, \infty)$ there exists a finite constant $A$ depending on $T$ 
	such that 
	
	\[
		\sup_{x\in\R^3}
		\E\left(|u(t\,,x) - u(t'\,,x)|^k\right) \le 
		\frac{A^k k^{k/2}}{%
		\left[\log_+(1/\|t-t'\|)\right]^{(\alpha-1) k/2}},  
	\]
	valid uniformly for all distinct $t,t'\in[0\,,T]$ and all real numbers $k\ge2$.
\end{proposition}

\begin{proof}
	We write
	$u(t+h\,,x) - u(t\,,x) = \cT_1(t\,,h\,,x) + \cT_2(t\,,h\,,x),$
	where
	\begin{equation}\label{TT12}\begin{split}
		\cT_1(t\,,h\,,x) &:= \int_{(t,t+h)\times\R^3} p_{t+h-s}(y-x) \sigma(u(s\,,y))\,\eta(\d s\,\d y),\\
		\cT_2(t\,,h\,,x) &:= \int_{(0,t)\times\R^3}
			\left[ p_{t+h-s}(y-x) - p_{t-s}(y-x)\right] \sigma(u(s\,,y))\,\eta(\d s\,\d y).
	\end{split}\end{equation}
	The proof readily follows from combining the subsequent
	Lemmas \ref{lem:T_1:temporal} through \ref{lem:T_2:temporal}.
\end{proof}

\begin{lemma}\label{lem:T_1:temporal}
	Recall $\cT_1(t\,,h\,,x)$ from \eqref{TT12}.
	If $\sigma$ is bounded,  then there exists a finite constant $A$
	such that
	\[
		\E\left(|\cT_1(t\,,h\,,x)|^k\right) \le \frac{ A^k k^{k/2}
		}{%
		\left[\log\left( 1/h\right)\right]^{k(\alpha-1)/2}},
	\] 
	uniformly for all $t>0$, $h\in(0\,,\e^{-2})$,
	$x\in\R^3$, and $k\in[2\,,\infty)$.
\end{lemma}

\begin{proof}
	A suitable form of the BDG inequality for martingales implies that
	\begin{align*}
		&\|\cT_1(t\,,h\,,x)\|_k^2\\
		&\le 4k\int_t^{t+h}\d s\int_{\R^3}\d y
			\int_{\R^3}\d y'\ p_{t+h-s}(y-x)p_{t+h-s}(y'-x)\, \|\sigma(u(s\,,y))
			\cdot\sigma(u(s\,,y'))\|_{k/2}
			\, f(y-y');
	\end{align*}
	see \cite{Khoshnevisan}. By the  boundedness of $\sigma$, 
	we obtain,
	 \begin{align*}
		\|\cT_1(t\,,h\,,x)\|_k^2 &\lesssim k
			\int_0^h\d s\int_{\R^3}\d y
			\int_{\R^3}\d y'\ p_s(y)p_s(y')  f(y-y')\\
		&\propto k 
			\int_0^h \d s\int_{\R^3}\d w\ \e^{-s\|w\|^2}  \widehat{f}(w)\\
		&= k  
			\int_{\R^3}\left( \frac{1-\e^{-h\|w\|^2}}{\|w\|^2}\right)
			\widehat{f}(w)\,\d w,  
	\end{align*}   
	thanks to Parseval's identity.
	Since $1-\exp(-a)\le \min(1\,,a)$ for all $a>0$, it then follows that
	\begin{equation}\label{cT2}
		\|\cT_1(t\,,h\,,x)\|_k^2 \lesssim k\
		\int_{\R^3} \min\left(\frac{1}{\|w\|^2}\,, h\right)\widehat{f}(w)\,\d w.
	\end{equation}  
	The integral can be considered separately in two parts: Where
	$\|w\|\le1/\sqrt{h}$ and where $\|w\|>1/\sqrt{h}$. The first part is estimated as follows: 
	\begin{align*}
		&h\int_{\|w\|\le1/\sqrt{h}}\widehat{f}(w)\,\d w
			= h\int_{\|w\|\le h^{-1/4}}\widehat{f}(w)\,\d w
			+h\int_{h^{-1/4}<\|w\|\le h^{-1/2}}\widehat{f}(w)\,\d w\\
		&\hskip1in\lesssim {h^{1/4}} + h\int_{h^{-1/4}\le \|w\|\le  h^{-1/2}}\frac{\d w}{\|w\|
			(\log\|w\|)^{\alpha}} \hskip1in\text{[see Theorem \ref{th:corr}]}\\
		&\hskip1in\propto {h^{1/4}} + h\int_{h^{-1/4}}^{h^{-1/2}}
			\frac{r\,\d r}{(\log r)^{\alpha}}\\
		&\hskip1in\asymp (\log(1/h))^{-\alpha}\\
		&\hskip1in\le \left[\log(1/h)\right]^{1-\alpha}.
	\end{align*}
	The second bound is handled similarly, viz.,
	\[
		\int_{\|w\|> 1/\sqrt{h}}\frac{\widehat{f}(w)}{\|w\|^2}\,\d w
		\lesssim\int_{\|w\|> 1/\sqrt{h}}\frac{\d w}{\|w\|^3 (\log\|w\|)^{\alpha} }
		\propto\int_{1/\sqrt h}^\infty\frac{\d r}{r(\log r)^{\alpha}}
		\propto\left[\log(1/h)\right]^{1-\alpha}.
	\]
	The lemma is a ready consequence of the preceding two displays and
	\eqref{cT2}.
\end{proof}

In order to estimate the quantity $\mathcal{T}_2(t\,,h\,,x)$ --- see \eqref{TT12} --- let us define
\begin{equation}\label{D}
	\mathcal{D}_r^{(h)}(a):=
	|p_{r+h}(a) - p_r(a)|\qquad
	\text{for all $r>0$ and $a\in\R^3$}.
\end{equation}

\begin{lemma}\label{lem:D1}
	For some universal constant $C>0$, it holds that
	\[
		\int_0^t\|\mathcal{D}^{(h)}_s\|^2_{L^1(\R^3)}  \,\d s
		\le C h \quad\text{for all $t, h>0$.}
	\]
\end{lemma}

\begin{proof}
	Because
	\[
		\dot{p}_t(x):=\frac{\partial}{\partial t} p_t(x)= 
		\frac{p_t(x)}{2t}\left[ \frac{\|x\|^2}{t}-3\right]\qquad\text{for all
		$t>0$ and $x\in\R^3$},
	\]
	we apply the fundamental theorem of calculus to see that
	\[
		\mathcal{D}^{(h)}_s(x) \le \int_s^{s+h}|\dot{p}_r(x)|\,\d r
		\lesssim\int_s^{s+h}\frac{p_r(x)}{r}\left(\frac{\|x\|^2}{r}+1\right)\d r.
	\]
	Integrate the preceding $[\d x]$ in order to see that
	\[
		\|\mathcal{D}^{(h)}_s\|_{L^1(\R^3)} =
		\int_{\R^3}\mathcal{D}^{(h)}_s(x)\,\d x
		\lesssim\int_s^{s+h}\frac{\d r}{r}
		= \log\left( 1 + \frac hs\right).
	\]
	Hence, by change of variable $u=h/s$,
	\[
		\int_0^t\|\mathcal{D}^{(h)}_s\|_{L^1(\R^3)}^2\,\d s \lesssim
		\int_0^{t}\left[\log\left( 1 + \frac hs\right)\right]^2\d s
		\le A h,
	\]
	where $A :=\int_0^\infty u^{-2} \left[\log(1+u)\right]^2\,\d u<\infty$.
	This proves the lemma. 
\end{proof}

\begin{lemma}\label{lem:T_2:temporal}
	Recall $\cT_2(t\,,h\,,x)$ from \eqref{TT12}. 
	If $\sigma$ is bounded, then there exists a finite constant $A$
	such that
	\[
		\E\left(|\cT_2(t\,,h\,,x)|^k\right) \le \frac{A^kk^{k/2}}{
		\left[\log\left( 1/h\right)\right]^{k(\alpha-1)/2}},
	\]
	uniformly  for all $t\in [0\,,T]$, $h\in(0\,,\e^{-2})$,
	$x\in\R^3$, and $k\in[2\,,\infty)$.
\end{lemma}

\begin{proof}
 	We begin as in the proof of the preceding lemma.
	Namely, we begin by observing
	that a suitable form of the BDG inequality for martingales implies that
	\begin{align*}
		&\|\cT_2(t\,,h\,,x)\|_k^2\\
		&\le 4k\int_0^t\d s\int_{\R^3}\d y
			\int_{\R^3}\d y'\ \mathcal{D}^{(h)}_{t-s}(y-x)\mathcal{D}^{(h)}_{t-s}(y'-x)\,
			\| \sigma(u(s\,,y))\cdot\sigma(u(s\,,y'))\|_{k/2}
			\, f(y-y'),
	\end{align*}
	By the boundedness of $\sigma$
	and the Cauchy--Schwarz inequality,
	\begin{align}\label{int:RHS}
		&\|\cT_2(t\,,h\,,x)\|_k^2 
			\lesssim  k 
			\int_0^t\d s\int_{\R^3}\d y
			\int_{\R^3}\d y'\ \mathcal{D}^{(h)}_s(y)\mathcal{D}^{(h)}_s(y') f(y-y'),
	\end{align}   
	where $\mathcal{D}_r^{(h)}(x)$ was defined earlier in \eqref{D};
	see the derivation of 
	\eqref{cT2}.
	Denote the triple integral in \eqref{int:RHS} by $I$. 
	The integral $I$ can be expressed as follows: 
	\begin{gather*}
		I = I_1 + I_2,\quad\text{where}\\
		I_1 := \int_0^t\d s\int_{\R^3}
			\int_{|y'-y|\leq \sqrt{h}} \d y' \d y\ \mathcal{D}^{(h)}_s(y)\mathcal{D}^{(h)}_s(y') f(y-y'); \\
		I_2 := \int_0^t\d s\int_{\R^3}
			\int_{|y'-y|\geq \sqrt{h}} \d y' \d y\ \mathcal{D}^{(h)}_s(y)\mathcal{D}^{(h)}_s(y') f(y-y').
	\end{gather*}
	Next, $I_1$ and $I_2$ are estimated separately, and in reverse order.	
	Theorem \ref{th:corr}  and Lemma \ref{lem:D1} together imply that
	\begin{equation}\label{est:I_2}\begin{split}
		I_2 &\le   
			\varphi(\sqrt{h})\int_0^t\d s\int_{\R^3}
			\int_{|y'-y|\geq \sqrt{h}} \d y' \d y\ \mathcal{D}^{(h)}_s(y)\mathcal{D}^{(h)}_s(y')\\
		&\lesssim h^{-1} \left[ \log (1/h)\right]^{-\alpha}
			\int_0^t \|\mathcal{D}^{(h)}_s\|^2_{L^1(\mathbb{R}^3)} \,\d s\\
		&\lesssim \left[ \log (1/h)\right]^{-\alpha}.
	\end{split}\end{equation}
	In order to estimate $I_1$, one can use the trivial inequality
	$\mathcal{D}_s^{(h)}(y')\le p_{s+h}(y')+p_s(y')$ in order to see that
	\begin{align*}
		I_1 &\le \int_0^t \d s \int_{\mathbb{R}^3} \d y \int_{|y'-y|\leq \sqrt{h}} \d y'\
			\mathcal{D}_s^{(h)}(y) \left(p_{s+h}(y')+p_s(y')\right) f(y-y') \\
		&\le \sum_{n=0}^{\infty} \varphi\left(2^{-n-1} \sqrt{h}\right)
			\int_0^t \d s \int_{\mathbb{R}^3} \d y 
			\int_{2^{-n-1} \sqrt{h} \leq |y-y'| \leq 2^{-n} \sqrt{h} }\d y'\ 
			\mathcal{D}_s^{(h)}(y) \left(p_{s+h}(y')+p_s(y')\right).
		\end{align*}
	Because $p_r(z) \lesssim  r^{-3/2}$  for all $z\in\R^3$ and $r>0$,
	and since $p_r$ is a probability density,
	one can estimate the $\d y'$-integral in the preceding display as follows:
	\begin{align*}
	\int_{2^{-n-1} \sqrt{h} \leq |y-y'| \leq 2^{-n} \sqrt{h}} (p_{s+h}(y') + p_{s}(y'))\d y'
		&\lesssim \left( \frac{(2^{-n-1} \sqrt{h})^3}{(s+h)^{3/2}} \wedge 1\right)
		+  \left( \frac{(2^{-n-1} \sqrt{h})^3}{s^{3/2}} \wedge 1\right)\\
		&\lesssim\left( \frac{(2^{-n-1} \sqrt{h})^3}{s^{3/2}} \wedge 1\right).
	\end{align*}
	Therefore,
	\begin{align*}
		I_1 &\lesssim \sum_{n=0}^{\infty} \varphi\left(2^{-n-1} \sqrt{h}\right) \int_0^t
			\left( \frac{(2^{-n-1} \sqrt{h})^3}{s^{3/2}} \wedge 1\right) 
			\, \d s   \\
		&\lesssim \sum_{n=0}^{\infty} \varphi(2^{-n-1} \sqrt{h}) 
			\left( \int_0^ {(2^{-n-1}\sqrt{h})^2} \d s + 
			\int_{(2^{-n-1}\sqrt{h})^2}^\infty  \frac{\left( 2^{-n-1} \sqrt{h}\right)^3}{s^{3/2}} \d s \right)\\
		&\lesssim h\sum_{n=0}^{\infty} \varphi\left( 2^{-n-1} \sqrt{h}\right)  2^{-2n}.
	\end{align*}
	Now apply Theorem \ref{th:corr} to see that
	\[
		I_1 \lesssim \sum_{n=0}^{\infty} \left[\log \left(\frac{2^{n+1}}{\sqrt{h}}\right)\right]^{-\alpha}
		\lesssim \sum_{n=0}^\infty\frac{1}{n^\alpha + [\log(1/h)]^\alpha}\\
		\lesssim \int_0^\infty \frac{\d q}{q^\alpha+[\log(1/h)]^\alpha}\\
		= B[\log(1/h)]^{1-\alpha},
	\]
	where $B:= \int_0^\infty [1+u^\alpha]^{-1}\,\d u<\infty$.
	The lemma follows from this and \eqref{est:I_2}.
\end{proof}



\section{The constant-coefficient case}\label{sec: additive noise}
So far, everything that was considered held for any $\alpha>1$.
From now on, we restrict the choice of the spatial correlation function $f$
further by assuming that $f$ comes from Theorem \ref{th:corr} in the
special case that
\begin{equation}\label{1:alpha:2} 
	1<\alpha<2.
\end{equation}
This assumption will be in place throughout the remainder of this paper, 
and used sometimes without mention.
	
In this section we study the [constant-coefficient] linearization of (SHE).
That is, we consider the stochastic partial differential equation,
\[
	\frac{\partial Z(t\,,x)}{\partial t} = \frac12 (\Delta Z)(t\,,x) + \eta(t\,,x),
\]
subject to $Z(0)\equiv 1$. As is well known,
the solution is the following centered Gaussian random field:
\[
	Z(t\,,x) := 1+ \int_{(0,t)\times\R^3} p_{t-s}(y-x)\,\eta(\d s\,\d y),
\]
as  the preceding Wiener integral has a finite variance. This can be seen
from an application of Lemma
\ref{lem:exist} with $\sigma\equiv 1$.

Recall the function $\varphi$ from \eqref{varphi}.
The elementary properties of Wiener integrals show us
that $Z$ is a centered Gaussian random field with covariance
\begin{align*}
	\Cov\left[ Z(t\,,x)\,,Z(t\,,x')\right] &= \int_0^t\d s\int_{\R^3}\d y\int_{\R^3}\d y'\
		p_s(x-y) p_s(x'-y')f(y-y')\\
	&= \int_0^t\d s\int_{\R^3}\d y\int_{\R^3}\d y'\
		p_s(x-y) p_s(x'-y')\varphi(\|y-y'\|),
\end{align*}
for all $t>0$ and $x,x'\in\R^3$.
In particular, it follows readily that $Z(t)$ is a
centered, stationary Gaussian random field --- indexed by $\R^3$ --- for every fixed $t>0$.

The following is the main result of this section.

\begin{proposition}\label{pr:max:Z}
	For every real number $t>0$ there exists a finite $K>1$ such that
	\[
		\P\left\{ \max_{{j\in \{1,\dots, N\}^3}}|Z(t\,,j/N)| \le
	 	\frac1K (\log N)^{1-(\alpha/2)}\right\}
		\le   \exp\left( - \frac{[\log N]^{2-\alpha}}{{K^2}(1+t)}\right)
		\qquad\text{for every $N\in\N$}.
	\]
\end{proposition}

Let us make a few remarks before we prove Proposition \ref{pr:max:Z}.
First, we record the following ready corollary of Proposition \ref{pr:max:Z},
the stationarity of $x\mapsto Z(t\,,x)$, and the restriction
\eqref{1:alpha:2} on $\alpha$.

\begin{corollary}\label{cor:discont:Z:1}
	For every $t>0$,
	\[
		\P\left\{ \sup_{x\in \mathcal{B} \cap \Q^3} Z(t\,,x)=\infty
		\text{ for every open ball $\mathcal{B}\subset\R^3$}\right\}=1.
	\]
\end{corollary}
Therefore, Fubini's theorem yields the following.

\begin{corollary}\label{cor:discont:Z:2}
	With probability one,
	\[
		\sup_{(t,x)\in \mathcal{C} \cap \Q^4} Z(t\,,x)=\infty
		\text{ for all open balls $\mathcal{C}\subset(0\,,\infty)\times\R^3$}.
	\]
\end{corollary}

The proof of Proposition \ref{pr:max:Z} hinges on an $L^2(\P)$-modulus of
continuity of $x\mapsto Z(t\,,x)$.

\begin{proposition}\label{pr:Z:modulus}
	Uniformly for all $t\ge0$ and $x,x'\in\R^3$,
	\[
		\frac{1-\e^{-t/2}}{\left[
		\log_+(1/\|x-x'\|)\right]^{\alpha-1}}\lesssim
		\E\left(|Z(t\,,x)-Z(t\,,x')|^2\right) \lesssim\frac{\e^t}{\left[
		\log_+(1/\|x-x'\|)\right]^{\alpha-1}}.
	\]
\end{proposition}

Proposition \ref{pr:Z:modulus} implies that $x\mapsto Z(t\,,x)$
is continuous in $L^2(\P)$, and hence in $L^p(\P)$ since the $L^p(\P)$
norm of a Gaussian random variable is controlled by its $L^2(\P)$ norm.
In particular, Doob's regularity theory implies that $x\mapsto Z(t\,,x)$
has a separable, in fact, Lebesgue measurable, version; 
see Chapter 5 of Khoshnevisan \cite{MPP}. After we establish
Proposition \ref{pr:Z:modulus} we always tacitly
refer to that separable version.

\begin{proof}[Proof of Proposition \ref{pr:Z:modulus}]
	By Parseval's identity,
	\begin{align*}
		\int_{\R^3}\d y\int_{\R^3}\d y'\ p_s(x-y) p_s(x'-y')f(y-y') &=
			\int_{\R^3}\d y\int_{\R^3}\d y'\ p_s(x-x'-y) p_s(y')f(y-y')\\
		&= \frac{1}{(2\pi)^3}\int_{\R^3} \exp\left( iz\cdot(x-x') - {s\|z\|^2} \right)
			\widehat{f}(z)\,\d z\\
		&=\frac{1}{(2\pi)^3}\int_{\R^3} \cos[z\cdot(x-x')] \e^{{-s\|z\|^2}}
			\widehat{f}(z)\,\d z.
	\end{align*}
	Therefore, an appeal to Lemma 4.1 of Foondun and Khoshnevisan \cite{FK2} yields
	\begin{equation}\label{TET}
		(1-\e^{-t/2}) \mathcal{T}
		\le \E\left(| Z(t\,,x)-Z(t\,,x')|^2\right)
		\le\e^{t/2}\mathcal{T},
	\end{equation}
	where
	\begin{equation}\label{T}
		\mathcal{T} := \frac{2}{(2\pi)^3}\int_{\R^3}\frac{1-\cos[z\cdot(x-x')]}{1+\|z\|^2}
		\widehat{f}(z)\,\d z.
	\end{equation}
	Thanks to Theorem \ref{th:corr},
	\begin{align*}
		\mathcal{T} &\asymp\int_{\R^3}\frac{1-\cos[z\cdot(x-x')]}{1+\|z\|^2}\,
			\frac{\d z}{1+\|z\| (\log_+\|z\|)^{\alpha}}\\
		&\asymp\int_{\R^3}\frac{1-\cos[z\cdot(x-x')]}{1+
			\|z\|^3 (\log_+\|z\|)^{\alpha}}\,\d z.
	\end{align*}
	Since $1-\cos\theta\le \min(1\,, \theta^2)$ for all $\theta\in\R$,
	\begin{equation}\label{G0}\begin{split}
		\mathcal{T} &\lesssim \int_{\R^3}
			\frac{\min\left( 1\,, \|z\|^2\|x-x'\|^2\right)}{1+
			\|z\|^3 (\log_+\|z\|)^{\alpha}}\,\d z\\
		&\lesssim \|x-x'\|^2 + \int_{\e\le \|z\|}G(\|z\|)\,\d z,
	\end{split}\end{equation}
	where
	\[
		G(r) := \frac{\min(1\,,{r^2}\|x-x'\|^2)}{r^3
		(\log r)^{\alpha}}\qquad\text{for all $r>\e$}.
	\]
	Integrate in spherical coordinates to find that
	\begin{equation}\label{G1}\begin{split}
		\int_{\|z\|>1/\|x-x'\|}G(\|z\|)\,\d z
			&\asymp \int_{1/\|x-x'\|}^\infty\frac{\d r}{%
			r (\log r)^{\alpha}}\\
		&=\int_{\log(1/\|x-x'\|)}^\infty\frac{\d s}{s^\alpha}\\
		&\propto \left[\log(1/\|x-x'\|)\right]^{1-\alpha}.
	\end{split}\end{equation}
	Similar computations yield the following:
	\begin{equation}\label{G2}\begin{split}
		\int_{\e\le\|z\|\le 1/\sqrt{\|x-x'\|}}G(\|z\|)\,\d z
			&\propto\|x-x'\|^2\int_{\e}^{1/\sqrt{\|x-x'\|}}
			\frac{r\,\d r}{(\log r)^{\alpha}}\\
		&\lesssim\|x-x'\|^2 \int_{\e}^{1/\sqrt{\|x-x'\|}}r\,\d r\\
		&\lesssim\|x-x'\|;
	\end{split}\end{equation}
	and
	\begin{equation}\label{G3}\begin{split}
		\int_{1/\sqrt{\|x-x'\|}\le\|z\|\le1/\|x-x'\|}G(\|z\|)\,\d z
			&\propto\|x-x'\|^2\int_{1/\sqrt{\|x-x'\|}}^{1/\|x-x'\|}
			\frac{r\,\d r}{(\log r)^{\alpha}}\\
		&\lesssim \|x-x'\|^2 \left[\log(1/\|x-x'\|)\right]^{1-\alpha}
			\int_{1/\sqrt{\|x-x'\|}}^{1/\|x-x'\|}r\,\d r\\
		&\lesssim\left[\log(1/\|x-x'\|)\right]^{1-\alpha}.
	\end{split}\end{equation}
	Therefore, we can combine \eqref{G1}, \eqref{G2}, and \eqref{G3},
	and plug the end result into \eqref{G0} to see that
	\begin{equation}\label{TUB}
		\mathcal{T} \lesssim \left[\log(1/\|x-x'\|)\right]^{1-\alpha}.
	\end{equation}
	This and \eqref{TET} together imply the upper bound for
	$\E(|Z(t\,,x)-Z(t\,,x')|^2)$.

	For the corresponding lower bound, we once again use \eqref{TET}
	and \eqref{T}. In this way we can show that
	\begin{align*}
		\mathcal{T} &\propto\int_{\R^3}
			\left(1-\cos[z\cdot(x-x')]\right)\frac{\widehat{f}(z)}{1+\|z\|^2}
			\,\d z\\
		&\ge\int_{\substack{z\in\R^3:\\\|z\|>1/\|x-x'\|}}
			\left(1-\cos[z\cdot(x-x')]\right)\frac{\widehat{f}(z)}{1+\|z\|^2}
			\,\d z\\
		&\gtrsim\int_{\substack{z\in\R^3:\\\|z\|>1/\|x-x'\|}}
			\frac{\widehat{f}(z)}{1+\|z\|^2}\,\d z;
	\end{align*}
	see Lemma 4.8 of Foondun and Khoshnevisan \cite{FK2} for an explanation of
	the last line. For $|x-x'|\leq \e^{-1}$ we apply Theorem \ref{th:corr} in order to deduce the following:
	\[
		\mathcal{T} \gtrsim
		\int_{1/\|x-x'\|}^\infty
		\frac{r\,\d r}{\left(1+r^2\right) ( \log(1/r) )^{\alpha}}
		\gtrsim\int_{1/\|x-x'\|}^\infty\frac{\d r}{r (\log(1/r))^{\alpha}}
		\propto \frac{1}{\left[\log(1/\|x-x'\|)\right]^{\alpha-1}}\,,
	\]
	and for the case $|x-x'| > \e^{-1}$, we merely write 
	\begin{equation*}
	 	\int_{\substack{z\in\R^3:\\\|z\|>1/\|x-x'\|}}
		\frac{\widehat{f}(z)}{1+\|z\|^2}\,\d z\ge
	 	\int_{\substack{z\in\R^3:\\\|z\|>\e}}
		\frac{\widehat{f}(z)}{1+\|z\|^2}\,\d z\,.
	\end{equation*}
 	Because of \eqref{TUB}, the preceding and \eqref{TET} together complete the task.
\end{proof}

Now we verify Proposition \ref{pr:max:Z}.

\begin{proof}[Proof of Proposition \ref{pr:max:Z}]
	Recall that $t>0$ is fixed, and define for every $x,x'\in\R^3$,
	\[
		d(x\,,x') := \sqrt{\E\left( \left|Z(t\,,x)-Z(t\,,x')\right|^2\right)}.
	\]
	Then, $d$ is the canonical metric that the Gaussian random field
	$x\mapsto Z(t\,,x)$ imposes on $\R^3$. And, in accord with Proposition
	\ref{pr:Z:modulus},
	\begin{equation}\label{d:Z}
		d(x\,,x') \asymp\left[\log_+\left(\frac{1}{\|x-x'\|}\right)\right]^{-(\alpha-1)/2},
	\end{equation}
	uniformly for all $x,x'\in\R^3$.

	For every $d$-compact set $A\subset\R^3$,
	let $\mathcal{N}_A(\,\cdot)$ denote the metric entropy of
	$A$ in the metric $d$; that is, for every $\varepsilon>0$,
	the quantity $\mathcal{N}_A(\varepsilon)$ denotes the minimum number
	of $d$-balls of radius $\varepsilon>0$ that are required to cover $A$.
	We have noted already that $x\mapsto Z(t\,,x)$ is a centered and stationary
	Gaussian process. Therefore, the theory of Dudley \cite{Dudley} and Fernique \cite{Fernique}
	[for a pedagogic account see Marcus and Rosen \cite{MR}] together
	imply that
	\begin{equation}\label{DF}
		\E\left( \sup_{x\in A}|Z(t\,,x)|\right)
		\asymp
		\int_0^{\text{\rm diam}(A)}\sqrt{\log_+\mathcal{N}_A(\varepsilon)}\,\d\varepsilon,
	\end{equation}
	uniformly for every compact set $A\subset\R^3$, where
	$\text{diam}(A):= \max_{a,b\in A}d(a\,,b)$ denotes the $d$-diameter of $A$
	and
	\[
		\E\left(\sup_{x\in A}|Z(t\,,x)|\right) :=
		\sup_{\substack{F\subset A:\\
		F\text{ is finite}}}\E\left(\max_{x\in F}|Z(t\,,x)|\right),
	\]
	which makes sense, thanks to separability.

	It is a noteworthy observation
	that the topology induced by the metric $d$ is Euclidean,
	thanks to \eqref{d:Z}. Therefore, \eqref{DF} holds for every compact set $A\subset\R^3$.

	We wish to apply the preceding to the finite set
	\[
		A := \left\{ \left(i_1/N\,,i_2/N\,,i_3/N\right):\
		0\le i_1,i_2,i_3\le N\right\}.
	\]
	According to \eqref{d:Z}, $\text{diam}(A)\asymp 1$ uniformly for all integers $N\ge1$.
	Moreover, for all $i,j\in\{0\,,\ldots,N\}^3$ and $N\ge 1$,
	\[
		d\left(i/N ~,~ j/N\right)
		\asymp\left[\log_+\left(\frac{N}{\|i-j\|}\right)\right]^{-(\alpha-1)/2}.
	\]
	In particular, there exists a finite constant $K>1$ such that for all $\varepsilon>0$:
	\begin{itemize}
		\item  If $d(i/N~,~j/N)\le\varepsilon$ then
			$\|i-j\|\le N\exp[-(K\varepsilon)^{-2/(\alpha-1)}]$; and
		\item If $N\exp[-(\varepsilon/K)^{-2/(\alpha-1)}] \le\|i-j\|$
			then $d(i/N~,~j/N)\le\varepsilon$.
	\end{itemize}
	Let $\delta(N)$ denote the smallest $\varepsilon>0$ such that
	$\varepsilon\ge d(i/N~,~j/N)>0$ for two distinct points $i, j\in NA$.
	The preceding remarks together imply that
	\[
		\delta(N) := \left[\log N\right]^{-(\alpha-1)/2}\qquad\text{for all $N\ge 1$},
	\]
	and
	\[
		\sqrt{\log\mathcal{N}_A(\varepsilon)}\asymp 
		\left\{\begin{array}{ll} 
			\varepsilon^{-1/(\alpha-1)}& \text{for all $\varepsilon\in
				\left(\delta(N)\,,\text{diam}(A)\right)$\,,} \\[1mm]
			(\log N)^{1/2} & 0 < \varepsilon \leq \delta(N).
		\end{array}\right. \qquad\]
	Now apply \eqref{DF} to see that
	\begin{equation}\label{E:max:Z}
		\E\left[\max_{1\le j\le N} |Z(t\,,j/N)|\right]
		\asymp \int_0^{\text{\rm diam}(A)}
		\frac{\d\varepsilon}{\varepsilon^{1/(\alpha-1)}}
		\asymp \left(\log N\right)^{(2-\alpha)/2},
	\end{equation}
	uniformly for all $N\ge1$. Finally, we apply Borell's inequality
	in order to see that for all $z>0$,
	\[
		\P\left\{ \left| \max_{x\in A}|Z(t\,,x)| - \E\left[ \max_{x\in A}|Z(t\,,x)|\right]
		\right| >  z \right\} \le 2\exp\left( - \frac{z^2}{2\sup_{x\in\R^3}\Var[Z(t\,,x)]}\right);
	\]
	see Borell \cite{Borell} and Sudakov and Tsirel'son \cite{ST}.
	This, \eqref{E:max:Z},
	and Remark \ref{rem:Var[Z]} together yield the proposition, after we make a judicious choice
	of $z$.
\end{proof}

\section{Local linearization}

For every space-time function $\phi:\R_+\times\R^3\to\R$,
and for every $\bm\varepsilon\in(0\,,\infty)^3$, define 
\[
	(\nabla_{\bm\varepsilon} \phi)(t\,,x) := \phi(t\,,x+\bm\varepsilon) - \phi(t\,,x).
\]
In other words, $\nabla_{\bm\varepsilon}$ is a sort
of discrete spatial gradient operator on a mesh of size
$\|\bm\varepsilon\|$. In particular, note that
\[
	(\nabla_{\bm\varepsilon} p)(t\,,x) = p_t(x+\bm\varepsilon)-p_t(x),
\]
for all $\bm\varepsilon\in(0\,,\infty)^3$ and $x\in\R^3$, where $p_t(\cdot)$ is the heat kernel,
as was defined in \eqref{p}.

We may also observe that $\nabla_{\bm\varepsilon}\phi$ makes sense equally well
when $\phi$ depends only on a spatial variable. In other words, whenever
$x\mapsto\phi(x)$ is a function on $\R^3$, 
\[
	(\nabla_{\bm\varepsilon}\phi)(x) = \phi(x+\bm\varepsilon)-\phi(x),
\]
for all $\bm\varepsilon,x\in\R^3$. 

In the next section we show that, under some additional assumptions on $\sigma$, 
the solution to \eqref{SHE} can be discontinuous at any given space-time point. 
The idea is that, in a strong sense,
\begin{equation}\label{approx}
	(\nabla_{\bm\varepsilon} u)(t\,,x) \approx \sigma(u(t\,,x)) (\nabla_{\bm\varepsilon} Z)(t\,,x)
	\quad\text{whenever $\bm\varepsilon\approx0$},
\end{equation}
for all $t>0$ and $x\in \R^3$. Hence, the gloal discontinuity of $Z(t\,,x)$
--- see Corollaries \ref{cor:discont:Z:1} and \ref{cor:discont:Z:2} ---
will force the local discontinuity of $u(t\,,x)$, as long as $\sigma(u(t\,,x))$
is not too small.
As part of this work, it will be shown that the error in the approximation 
\eqref{approx} to $\nabla_{\bm\varepsilon} u$ does not affect the discontinuity of the
term $\sigma(u)\times \nabla_{\bm\varepsilon} Z$. 
The following result makes this assertion more precise. 

\begin{proposition}\label{pr:localization}
	Assume that $\sigma$ is bounded. For any $T>0$,  there exist positive and finite constants $A_T$ and $\varepsilon_*<1$ such that  
	\[
		\E\left(\left| (\nabla_{\bm\varepsilon} u) (t\,,x) - \sigma(u(t\,,x))
		(\nabla_{\bm\varepsilon} Z) (t\,,x) \right|^k\right) \le
		\frac{{(A_T \sigma_0\sqrt{k})^k
		}}{%
		[\log(1/\|\bm\varepsilon\|)]^{3k(\alpha-1)/4}},
	\] 
	uniformly for all $(t\,,x\,,k)\in[0,T]\times\R^3\times[2\,,\infty)$
	and $\bm\varepsilon\in(0\,,\infty)^3$ that satisfy
	$\|\bm\varepsilon\|<\varepsilon_*$, where $\sigma_0$ is the bound for $\sigma$, i.e., $|\sigma(z)|\le \sigma_0$ for all $z\in \R$. 
\end{proposition}

The proof is somewhat long, and will be presented shortly. But first, let us 
make a few remarks on the content of Proposition \ref{pr:localization}.

According to Propositions \ref{pr:u(x)-u(x')}, for every $k\ge 2$ and $T>0$,
\[
	\left\| (\nabla_{\bm\varepsilon} u)(t\,,x)\right\|_k \lesssim
	[\log(1/\|\bm\varepsilon\|)]^{-(\alpha-1)/2},
\]
uniformly for all $(t\,,x)\in[0\,,T]\times\R^3$ and $\|\bm\varepsilon\|>0$ sufficiently
small.
This very inequality can be applied with $k$ replaced by $2k$ and
$\sigma$ by the constant function $1$ in order to yield the following: For all $k\ge 2$ and $T>0$,
\[
	\left\| (\nabla_{\bm\varepsilon} Z)(t\,,x)\right\|_{2k} \lesssim
	[\log(1/\|\bm\varepsilon\|)]^{-(\alpha-1)/2},
\]
uniformly for all $(t\,,x)\in[0\,,T]\times\R^3$  and $\bm{\varepsilon}\in\R^3\setminus\{0\}$
such that $\|\bm\varepsilon\|$ is sufficiently small.
Theorem \ref{th:LyapunovExp}
and the Lipschitz continuity of $\sigma$ together imply that $\|\sigma(u(t\,,x))\|_{2k}$
is bounded, for every $k\ge 2$ and $T>0$, uniformly over all
$(t\,,x)\in[0\,,T]\times\R^3$. One can conclude from this discussion,
and the Cauchy--Schwarz inequality, that for
all $k\ge2$ and $T>0$,
\[
	\max\left\{ \left\| (\nabla_{\bm\varepsilon} u)(t\,,x)\right\|_k \,,
	\left\| \sigma(u(t\,,x))(\nabla_{\bm\varepsilon} Z)(t\,,x)\right\|_k \right\}\lesssim
	[\log(1/\|\bm\varepsilon\|)]^{-(\alpha-1)/2},
\]
uniformly for all $(t\,,x)\in[0\,,T]\times\R^3$  and $\|\bm\varepsilon\|>0$ sufficiently
small. If this bound were proved to 
be sharp [it can be, in some cases], then Proposition \ref{pr:localization}
is telling us that, although 
$\mathcal{A}_1 := (\nabla_{\bm\varepsilon} u)(t\,,x)$ and
$\mathcal{A}_2 := \sigma(u(t\,,x))(\nabla_{\bm\varepsilon} Z)(t\,,x)$ 
are both quite small in $L^k(\P)$ norm, their difference $\mathcal{A}_1-\mathcal{A}_2$ 
is smaller still. This is a quantitative way to say that 
the locally-linearized form $\mathcal{A}_2$ is a very good approximation
to the discrete gradient $\mathcal{A}_1$ of the solution to \eqref{SHE}.
This general idea has recently played various roles in SPDEs;
see, for example Hairer \cite{Hairer,Hairer1} and Hairer and Pardoux \cite{HP}, where this sort of 
local linearization is sometimes referred
to as a ``jet expansion,'' and Foondun, Khoshnevisan, and Mahboubi \cite{FKM}
and Khoshnevisan, Swanson, Xiao, and Zhang \cite{KSXZ}, where this sort of
local linearization is used to analyse the local structure of the solution
to parabolic SPDEs that are much nicer than those that appear here.

Let us conclude this section with the following.

\begin{proof}[Proof of Proposition \ref{pr:localization}]
	Let us first introduce some notation.
	
	For every $\varepsilon>0$ let
	\begin{equation}\label{beta:gamma}
		\beta_\varepsilon:= \exp\left(-\sqrt{\log(1/\varepsilon)}\right)
		\quad\text{and}\quad
		\gamma_\varepsilon := (16\beta_\varepsilon)^{1/4} =
		2\exp\left( -\tfrac14\sqrt{\log(1/\varepsilon)}\right).
	\end{equation}
	As notational advice, let us point out that here and throughout,
	$\varepsilon>0$ denotes a typically-small scalar and should not be confused
	with $\bm{\varepsilon}\in(0\,,\infty)^3$ which is a 3-vector that typically
	has small norm.
	
	For all $t,\varepsilon>0$ and $x\in\R^3$ define 
	\begin{equation}\label{B}
		B(x\,,t\,,\varepsilon) = \left[ \left(t-\beta_\varepsilon\right)_+ \,, t\right] 
		\times \prod_{i=1}^3
		[x_i-\gamma_\varepsilon\,, x_i+ \gamma_\varepsilon]
	\end{equation} 
	to be a suitably-chosen, 4-dimensional, space-time box with ``center'' $(t\,,x)$.

	From now on we choose and fix a real number $\Xi>1$ and consider
	an arbitrary $\bm{\varepsilon}\in(0\,,\infty)^3$ that satisfies
	\[
		\Xi^{-1}\varepsilon\le \|\bm{\varepsilon}\|\le \Xi\varepsilon.
	\]
	
	Let us consider the following decomposition of $\nabla_{\bm\varepsilon} u$, valid
	thanks to \eqref{mild}
	\[
		(\nabla_{\bm\varepsilon} u)(t\,,x)-\sigma(u(t\,,x))(\nabla_{\bm\varepsilon} Z)(t\,,x) 
		= I_{11} - I_{12} + I_{21} - I_{22},
	\]
	where $I_{ij} = I_{ij}(t\,,x\,,\varepsilon)$ is defined for all
	$i,j=1,2$ as follows:
	\begin{align*}
		I_{11} &:=\int_{B(x,t,\varepsilon)^{c}} (\nabla_{\bm\varepsilon} p)(t-s\,,x-y) 
			\sigma(u(s\,,y))\, \eta(\d s\,\d y);\\
		I_{12} &:= \sigma(u(t\,,x))\cdot\int_{B(x,t,\varepsilon)^c} 
			(\nabla_{\bm\varepsilon} p)(t-s\,,x-y)
			\, \eta(\d s\,\d y);\\
		I_{21} &:=\int_{B(x,t,\varepsilon)} (\nabla_{\bm\varepsilon} p)(t-s\,,x-y) 
			\left[ \sigma(u(s\,,y))-\sigma\left(u\left(\left(t-\beta_\varepsilon\right)_+\,,x
			\right)\right)\right]\eta(\d s\,\d y);
			\text{ and}\\
		I_{22} &:= \left[\sigma(u(t\,,x))-\sigma\left(u\left(\left(t-\beta_\varepsilon\right)_+\,,x
			\right)\right)\right] \cdot
			\int_{B(x,t,\varepsilon)} (\nabla_{\bm\varepsilon} p)(t-s\,,x-y) \,\eta(\d s\,\d y).
	\end{align*}
	The $L^k(\P)$-norms of the $I_{ij}$'s are estimated next. The computations
	are somewhat long and tedious. Therefore, they are presented in five separate steps.  
	\\
	

	\noindent\textbf{Step 1. A comparison estimate for $I_{11}$.}
	In the second step of the proof we establish an inequality that compares
	the moments of the random variable $I_{11}$ to moments of a certain mean-zero
	Gaussian random variable; see \eqref{S2} below.
	
	A suitable formulation
	of the Burkholder--Davis--Gundy inequality \cite{Khoshnevisan} 
	implies that
	\begin{align*}
		\|I_{11}\|_k^2 &\le 4k
			\iiint \mathscr{D}(t-s\,,x-y\,,x-y')
			\left\| \sigma(u(s\,,y))\cdot\sigma(u(s\,,y'))\right\|_{k/2}  f(y-y')\,\d s\,\d y\,\d y'\\
		& \lesssim k\sigma_0^2 \iiint \left| (\nabla_{\bm\varepsilon} p)(t-s\,,x-y) 
			(\nabla_{\bm\varepsilon} p)(t-s\,,x-y')\right|
			f(y-y')\,\d s\,\d y\,\d y',
	\end{align*}
	where 
	$\mathscr{D}(r\,,a\,,a') := | (\nabla_{\bm\varepsilon} p)(r\,,a) 
	(\nabla_{\bm\varepsilon} p)(r\,,a') |,$
	and the triple integrals are computed over all points
	\begin{equation}\label{B2}
		(s\,,y\,,y')\not\in B_2(x\,,t\,,\varepsilon) := 
		\left[\left(t-\beta_\varepsilon\right)_+\,,t\right]\times
		\prod_{i=1}^3[x_i-\gamma_\varepsilon\,,x_i+\gamma_\varepsilon]^2,
	\end{equation}
	If $X$ is a random variable with the standard normal distribution, then
	$\|X\|_k\asymp\sqrt{k}$ uniformly for all $k\ge2$. This fact and
	the previous inequality for $\|I_{11}\|_k^2$ together yield
	\begin{equation}\label{S2}
		\|I_{11}\|_k \lesssim \sqrt{k}\sigma_0\left\|\int_{B(x,t,\varepsilon)^c} 
		|(\nabla_{\bm\varepsilon}  p)(t-s\,,x-y) |\,\eta(\d s\,\d y)\right\|_2,
	\end{equation} 
	valid uniformly for all $(t\,,x\,,k)\in[0,T]\times\R^3\times[2\,,\infty)$ and $\varepsilon>0$,
	where $B(x\,,t\,,\varepsilon)$ was defined in \eqref{B}.\\
	
	\noindent\textbf{Step 2. A Gaussian moment Estimate.}
	Next we develop a moment inequality for the Gaussian stochastic integral on the right-hand
	side of \eqref{S2}; the precise statement can be found in \eqref{S3} below. Since we are only interested in the behavior when $\varepsilon \to 0$, we will assume that $\beta_{\varepsilon}< t$ from now on. 
	
	By Minkowski's inequality,
	\begin{equation}\label{Q1Q2}
		\left\|\int_{B(x,t,\varepsilon)^c} 
		|(\nabla_{\bm\varepsilon}  p)(t-s\,,x-y) |\,\eta(\d s\,\d y)\right\|_2
		\le Q_1^{1/2} + Q_2^{1/2},
	\end{equation}
	where:
	\begin{align*}
		Q_1 &:= \E\left(\left|\int_0^{t-\beta_\varepsilon }
			\int_{\R^3} 
			|(\nabla_{\bm\varepsilon}  p)(t-s\,,x-y) |\,\eta(\d s\,\d y)\right|^2\right);
			\text{ and}\\
		Q_2 &:=\E\left(\left|\int_{t-\beta_\varepsilon}^t
			\int_{[x-\gamma_\varepsilon , x+ \gamma_\varepsilon ]^c} 
			|(\nabla_{\bm\varepsilon}  p)(t-s\,,x-y) |\,\eta(\d s\,\d y)\right|^2\right),
	\end{align*}
	where $[a\,,b] := \prod_{i=1}^3[a_i\,,b_i]$
	for all $a,b\in\R^3$.
	
	After one or two changes of variables [$t-s\to s$, $y-x\to y$],
	\begin{align*}
		Q_1 &= \E\left(\left|\int_{\beta_\varepsilon }^t
			\int_{\R^3} 
			|(\nabla_{\bm\varepsilon}  p)(s\,,y) |\,\eta(\d s\,\d y)\right|^2\right)\\
		&= \int_{\beta_\varepsilon }^t\d s
			\int_{(\R^3)^2}\d y\, \d y'\
			|(\nabla_{\bm\varepsilon}  p)(s\,,y) \cdot (\nabla_{\bm\varepsilon}  p)(s\,,y') |f(y-y')\\
		&= Q_{11} + Q_{12},
	\end{align*}
	where
	\begin{align*}
		Q_{11} &:= \int_{\beta_\varepsilon }^t\d s
			\int_{\substack{y,y'\in\R^3:\\ \|y-y'\|\ge \varepsilon}
			} \d y\,\d y'\
			|(\nabla_{\bm\varepsilon}  p)(s\,,y) \cdot   (\nabla_{\bm\varepsilon}  p)(s\,,y') |
			\varphi(\|y-y'\|);\text{ and}\\
		Q_{12} &:= \int_{\beta_\varepsilon }^t\d s
			\int_{\substack{y,y'\in\R^3:\\ \|y-y'\|< \varepsilon}
			} \d y\,\d y'\
			|(\nabla_{\bm\varepsilon}  p)(s\,,y) \cdot   (\nabla_{\bm\varepsilon}  p)(s\,,y') |
			\varphi(\|y-y'\|);
	\end{align*}
	see \eqref{varphi} for the definition of $\varphi$.
	
	According to \eqref{beta:gamma}, $\beta_\varepsilon > \varepsilon^2$ for all
	$\varepsilon>0$ sufficiently small. Therefore, \eqref{L1} and the monotonicity
	of $\varphi$ [Theorem \ref{th:corr}] together imply that
	\begin{equation}\label{Q11}\begin{split}
		Q_{11} &\le  \varphi(\varepsilon) \int_{\beta_\varepsilon }^t\d s 
			\int_{(\R^3)^2} \d y\,\d y'\
			|(\nabla_{\bm\varepsilon} p)(s\,,y) \cdot  (\nabla_{\bm\varepsilon} p)(s\,,y')|\\
		&\lesssim \varphi(\varepsilon) \int_{\beta_\varepsilon }^t 
			\left( \frac{\varepsilon}{\sqrt{s}}\wedge 1\right)^2 \d s\\
		&\lesssim  \varphi(\varepsilon) \varepsilon^2 \log\left( t/\beta_\varepsilon\right)\\
		&\lesssim [\log(1/\varepsilon)]^{(1/2)-\alpha}.
	\end{split}\end{equation}
	
	One can estimate $Q_{12}$ using the same technique that was used in
	the proof of Proposition \ref{pr:u(x)-u(x')}. More specifically, we
	proceed as follows: For all $\varepsilon>0$ sufficiently small,
	\begin{align*}
		Q_{12} &= \sum_{n=0}^{\infty} \int_{\beta_\varepsilon }^t \d s
			\int_{\substack{y,y'\in(\R^3)^2:\\2^{-n-1}\varepsilon \leq|y-y'| \leq 2^{-n} \varepsilon}}
			\d y\,\d y'\
			|(\nabla_{\bm\varepsilon} p)(s\,,y) \cdot  (\nabla_{\bm\varepsilon} p)(s\,,y')| \varphi(\|y-y'\|)\\
		&\le \sum_{n=0}^{\infty} \varphi\left(2^{-n-1} \varepsilon\right)
			\int_{\beta_\varepsilon }^t \left( \frac{\varepsilon}{\sqrt{s}} \wedge 1\right) 
			\left( \frac{2^{-n} \varepsilon}{\sqrt{s}}\wedge 1\right)^3 \d s 
			&\text{[see \eqref{L1}]}\\
		&= \varepsilon^4 \sum_{n=0}^{\infty} 8^{-n}\varphi\left(2^{-n-1} \varepsilon\right)
			\int_{\beta_\varepsilon }^t \frac{\d s}{s^2}
			&\text{[since $\beta_\varepsilon>\varepsilon^2$]}\\
		&\le \frac{\varepsilon^4}{\beta_\varepsilon}
			\sum_{n=0}^{\infty} 8^{-n}\varphi\left(2^{-n-1} \varepsilon\right).
	\end{align*}
	An appeal to Theorem \ref{th:corr} and \eqref{beta:gamma} yields
	\[
		Q_{12} \lesssim\varepsilon^2\e^{\sqrt{\log(1/\varepsilon)}}
		\sum_{n=0}^{\infty} 2^{-n}\left|
		\log\left( \frac{2^{n+1}}{\varepsilon}\right)\right|^{-\alpha}
		\lesssim \varepsilon^2\e^{\sqrt{\log(1/\varepsilon)}}
		[\log(1/\varepsilon)]^{1-\alpha},
	\]
	by an integral test. This inequality and \eqref{Q11}
	together yield
	\begin{equation}\label{Q1}
		Q_1 \lesssim[\log(1/\varepsilon)]^{(1/2)-\alpha},
	\end{equation}
	valid uniformly for all $(t\,,x)\in\R_+\times\R^3$ and all
	sufficiently small $\varepsilon>0$.
	
	In order to bound $Q_2$, we change variables and then use the simple inequality,
	\[
		\left| (\nabla_{\bm\varepsilon} p)(s\,,w)\right| \le 
		p_s(w+\bm\varepsilon) + p_s(w),
	\]
	in order to deduce that for all $\varepsilon$ sufficiently small,
	\begin{align*}
		Q_2 &\le \int_0^{\beta_\varepsilon}\d s
			\int_{[-\gamma_\varepsilon,\gamma_\varepsilon]^c}\d y
			\int_{[-\gamma_\varepsilon,\gamma_\varepsilon]^c}\d y'\
			\left[ p_s(y+\bm{\varepsilon}) + p_s(y)\right]
			\left[p_s(y'+\bm{\varepsilon}) + p_s(y')\right]f(y-y')\\
		&\le 4\int_0^{\beta_\varepsilon}\d s
			\int_{[-\gamma_\varepsilon/2,\gamma_\varepsilon/2]^c}\d y
			\int_{[-\gamma_\varepsilon/2,\gamma_\varepsilon/2]^c}\d y'\
			p_s(y) p_s(y')f(y-y') \\
		&\le 4\int_0^{\beta_\varepsilon}\d s
			\int_{[-\gamma_\varepsilon/2,\gamma_\varepsilon/2]^c}\d y\
			p_s(y) (p_s*f)(y).
	\end{align*} 
	Because $p_s$ and $f$ are both positive semi-definite,  so
	is $p_s*f$. Moreover, $p_s*f$ is continuous and bounded.
	Therefore, elementary facts about positive definite functions
	tell us that $p_s*f$ is maximized at the origin. In this way we find that
	\[
		Q_2 \lesssim\int_0^{\beta_\varepsilon} (p_s*f)(0)\,\d s
		\int_{[-\gamma_\varepsilon/2,\gamma_\varepsilon/2]^c}
		p_s(y) \,\d y
		\lesssim\int_0^{\beta_\varepsilon} \exp\left( -\frac{\gamma_\varepsilon^2}{{8s}}\right)
		(p_s*f)(0)\,\d s,
	\]
	since $\P\{\|X\|>R\}\lesssim\exp(-R^2/2)$ for all $R>0$ when $X$ is a 
	3-vector of i.i.d.\ standard normal
	random variables. An appeal to Lemma \ref{lem:f} now yields
	\begin{align*}
	Q_2 &\lesssim  \int_0^{\beta_\varepsilon} \exp\left( -\frac{\gamma_\varepsilon^2}{{8s}}\right)
			\frac{\d s}{s[\log(1/s)]^\alpha}\\
			&\leq \exp\left(-\frac{\gamma^2_\varepsilon}{ 8\beta_\varepsilon} \right) \int_0^{\beta_\varepsilon} \frac{\d s}{s[\log(1/s)]^\alpha} \, \d s\\
			&\lesssim \exp\left[ -\frac{1}{2} \exp\left( \frac{1}{2} \sqrt{\log (1/\varepsilon)} \right) \right] \left(\log (1/\varepsilon) \right)^{(1-\alpha)/2}\\
			&\lesssim  [\log (1/\varepsilon) ]^{(1/2)-\alpha}.   
	\end{align*} 
	thanks  to the definition \eqref{beta:gamma} of $\gamma_\varepsilon$
	and $\beta_\varepsilon$. It is easy to deduce from the preceding that
	$\limsup_{\varepsilon\downarrow0}\sqrt{\beta_\varepsilon}\log Q_2
	\le -\tfrac12$, and hence for every $\kappa>0$,
	\(
		Q_2 \lesssim [\log(1/\varepsilon)]^{-\kappa},
	\)
	uniformly for all $\varepsilon>0$ sufficiently small [with room to spare].
	This inequality, \eqref{Q1Q2} and \eqref{Q1} together 
	accomplish the main objective of Step 3; namely, they imply that there exists $\varepsilon_0\in(0\,,1)$ such that for some constant $A_T>0$, 
	\begin{equation}\label{S3}
		\left\|\int_{B(x,t,\varepsilon)^c} 
		|(\nabla_{\bm\varepsilon}  p)(t-s\,,x-y) |\,\eta(\d s\,\d y)\right\|_2 \lesssim A_T
		[\log(1/\varepsilon)]^{(1/4)-\alpha/2},
	\end{equation}
	for all $(t\,,x)\in[0,T]\times\R^3$ and $\varepsilon\in(0\,,\varepsilon_0)$.\\

	\noindent\textbf{Step 3. Estimates for $I_{11}$ and $I_{12}$.}
	It is now easy to find suitable estimates for the $L^k(\P)$ norm of 
	$I_{11}$ and $I_{12}$. The requisite bounds will appear in
	\eqref{S4.1} and \eqref{S4.2} below. 
	
	First of all we simply combine \eqref{S2} with \eqref{S3} to obtain
	the following estimate for $I_{11}$: 
	\begin{equation}\label{S4.1}
		{\| I_{11}\|_k \lesssim A_T\sqrt k \sigma_0
		[\log(1/\varepsilon)]^{(1/4)-\alpha/2}, }
	\end{equation}
	valid uniformly for all $(t\,,x\,,k)\in\R_+\times\R^3\times[2\,,\infty)$
	and $\varepsilon\in(0\,,\varepsilon_0)$.
	
	Next, we estimate the $L^k(\P)$ norm of $I_{12}$ as follows: By the Cauchy--Schwarz inequality,
	\begin{equation}\label{I_12:12}
		\| I_{12}\|_k \le \| \sigma(u(t\,,x)) \|_{2k}\cdot
		\left\| \int_{B(x,t,\varepsilon)^c} 
		(\nabla_{\bm\varepsilon} p)(t-s\,,x-y)
		\, \eta(\d s\,\d y) \right\|_{2k}.
	\end{equation}
	Because $\sigma$ is bounded, and 
	recall that if $X$ has a standard normal
	distribution, then  $\|X\|_k\asymp\sqrt k$, uniformly for all $k\ge 2$. 
	Therefore, the second quantity on the right-hand side of \eqref{I_12:12} can be
	bounded as follows:
	\begin{align*}
		\left\| \int_{B(x,t,\varepsilon)^c} 
			(\nabla_{\bm\varepsilon} p)(t-s\,,x-y)
			\, \eta(\d s\,\d y) \right\|_{2k}
			&\lesssim \sqrt{k}\sigma_0\left\| \int_{B(x,t,\varepsilon)^c} 
			(\nabla_{\bm\varepsilon} p)(t-s\,,x-y)
			\, \eta(\d s\,\d y) \right\|_2\\
		&\lesssim A_T\sqrt{k}\sigma_0[\log(1/\varepsilon)]^{(1/4)-\alpha/2},
	\end{align*}
	thanks to Step 2; see \eqref{S3}. We can combine the preceding inequalities to 
	deduce the following estimate for $I_{12}$: There exists $B\in(0\,,\infty)$ such that
	\begin{equation}\label{S4.2}
		\| I_{12} \|_k \lesssim A_T\sqrt{k}\sigma_0
		[\log(1/\varepsilon)]^{(1/4)-\alpha/2}, 
	\end{equation}
	uniformly for all $(t\,,x\,,k)\in[0,T]\times\R^3\times[2\,,\infty)$ and
	$\varepsilon\in(0\,,\varepsilon_0)$.\\
	
	\noindent\textbf{Step 4. Estimates for $I_{21}$ and $I_{22}$.}
	In this step we derive a bound for the moments of
	$I_{21}$ and $I_{22}$; the end results are \eqref{S5.1} and \eqref{S5.2} below.
	
	Let us recall the sets $B_2(x\,,t\,,\varepsilon)$ from \eqref{B2}. 
	By a suitable application of the Burkholder--Davis--Gundy inequality,
	\begin{equation}\label{pre:I_21}
		\| I_{21} \|_k^2 \le 4k \int_{B_2(x,t,\varepsilon)}
		\mathcal{P}(s\,,y)\mathcal{P}(s\,,y')\mathcal{U}(s\,,y)\mathcal{U}(s\,,y')
		f(y-y')\, \d s\,\d y\,\d y',
	\end{equation}
	where
	\[
		\mathcal{P}(s\,,w) := \left|(\nabla_{\bm\varepsilon} p)(t-s\,,{x-w})\right|
		\quad\text{and}\quad
		\mathcal{U}(s\,,w) := \left\| \sigma(u(s\,,w)) -
		\sigma\left(u\left(\left(t-\beta_\varepsilon\right)_+\,,x
		\right)\right)\right\|_k,
	\]
	for all $(s\,,w)\in B(x\,,t\,,\varepsilon)$.
	Of course, the functions $\mathcal{P}$ and $\mathcal{U}$ depend also
	on the variables $(t\,,x\,,\varepsilon)$, but this dependency
	is not immediately relevant 
	 to the discussion.
	
	Because of the Lipschitz condition of $\sigma$, Propositions
	\ref{pr:u(x)-u(x')} and \ref{pr:u(t)-u(t')} together imply that when $\varepsilon$ is sufficiently small, 
	\begin{align*}
		\mathcal{U}(s\,,w) &\lesssim \left\| u(s\,,w) -
			u\left(\left(t-\beta_\varepsilon\right)_+\,,w\right)\right\|_k + 
			\left\| u\left(\left(t-\beta_\varepsilon\right)_+\,,w\right) - 
			u\left(\left(t-\beta_\varepsilon\right)_+\,,x\right)\right\|_k\\
		&\lesssim A_T \sigma_0 \sqrt{k}
			\left[ \left| \log (\beta_\varepsilon) \right|^{-(\alpha-1)/2}
			+\left|\log(2\gamma_\varepsilon)\right|^{-(\alpha-1)/2}\right]\\
		&\lesssim A_T \sigma_0 \sqrt{k}[\log(1/\varepsilon)]^{-(\alpha-1)/4},
	\end{align*}
	uniformly for all $(s\,,w)\in B(x\,,t\,,\varepsilon)$,
	and $1/\infty:=0$ to account for the possibilities $s=t-\beta_\varepsilon$
	and $w=x$. Consequently, \eqref{pre:I_21} yields
	\begin{align*}
		\|I_{21}\|_k^2 &\lesssim \frac{A_T^2 \sigma_0^2 k}{[\log(1/\varepsilon)]^{(\alpha-1)/2}}
		\int_{B_2(x,t,\varepsilon)}\mathcal{P}(s\,,y)\mathcal{P}(s\,,y')
		f(y-y')\,\d y\,\d y'\, \d s \\
		&\lesssim \frac{A_T^2 \sigma_0^2 k}{[\log(1/\varepsilon)]^{(\alpha-1)/2}}
		\int_0^t \int_{\R^3} \int_{\R^3}\mathcal{P}(s\,,y)\mathcal{P}(s\,,y') f(y-y')\,\d y\,\d y'\, \d s \\
		& \lesssim \frac{A_T^2 \sigma_0^2 k }{[\log (1/\varepsilon)]^{3(\alpha-1)/2}}\,,
	\end{align*}
	the last line follows from \eqref{CASE1} and \eqref{CASE2} [simply apply the latter two inequalities
	with $\|z\|=\varepsilon$, for instance].
	This readily yields the following,
	with room to spare:
	There exist finite and positive constants $A$ and $\varepsilon_1<1$ such that
	\begin{equation}\label{S5.1}
		\|I_{21}\|_k \lesssim A_T\sigma_0\sqrt{k}  
		[\log(1/\varepsilon)]^{3(1-\alpha)/4}, 
	\end{equation}
	uniformly for all $(t\,,x\,,k)\in\R_+\times\R^3\times[2\,,\infty)$ and
	$\varepsilon\in(0\,,\varepsilon_1)$.
	
	Finally, we obtain $\|I_{22}\|_k$ from \eqref{S5.1}, using the Cauchy--Schwarz inequality in the
	same way that $I_{12}$ was derived from $I_{11}$ in Step 4, in order to obtain
	\begin{equation}\label{S5.2}
	{	\|I_{22}\|_k \lesssim A_T \sigma_0 \sqrt{k} 
		[\log(1/\varepsilon)]^{3(1-\alpha)/4},  }
	\end{equation}
	uniformly for all $(t\,,x\,,k)\in[0,T]\times\R^3\times[2\,,\infty)$ and
	$\varepsilon\in(0\,,\varepsilon_1)$.\\
	
	\noindent\textbf{Step 5. Conclusion of proof.} The proposition follows
	from an application of Minkowski inequality, using the results 
	\eqref{S4.1} through \eqref{S5.2} of Steps 1 through 5.
\end{proof}



\section{Proof of Theorems \ref{th:conditional} and \ref{th:local}}

The groundwork for the proof of the main results of the paper
has been laid. We now are ready to prove the main results of the paper,
which we do in order.

\begin{proof}[Proof of Theorem \ref{th:conditional}]
	We start by observing that 
	\begin{equation}\label{u>0}
		\P\{ u(t\,,x)>0\}>0\qquad\text{for every $t>0$ and $x\in \R^3$.}
	\end{equation}
	This is because $\E[u(t\,,x)]=1$, as can be deduced from \eqref{mild}.
	
	Next we observe that one can reduce the scope of the problem to the case
	that $\sigma(z)\ge0$ for all $z\in\R$ without incurring any loss in generality.
	This is because $\sigma$ is continuous and crosses zero at --- and only at --- the origin.
	A second appeal to the assumption $\sigma^{-1}\{0\}=0$ reduces the problem
	to proving the following:
	\begin{equation}\label{cond}
		\P \left(\left.\sup_{y \in B(x,r)} u(t\,,y)=\infty
		\text{ for all $r>0$}\ \right|\, \sigma(u(t\,,x))>0 \right)=1,
	\end{equation}
	for every $(t\,,x)\in(0\,,\infty)\times\R^3$, where the open ball $B(x\,,r)$ was defined in
	\eqref{ball}.
	
	Owing to \eqref{u>0}, we can find two finite numbers $0<A<B$ such that
	\begin{equation}\label{eq:A, B}
		\P\left\{ \sigma(u(t\,,x))\in[A\,,B] \right\}>0.
	\end{equation}
	We plan to prove that the following holds for every such pair $(A\,,B)$ of real numbers
	that satisfy \eqref{eq:A, B}:
	\begin{equation}\label{gobs}
		\lim_{r \to 0} \P \left(\left.\sup_{y \in B(x,r)} u(t\,,y)=\infty 
		\ \right|\, \sigma(u(t\,,x))\in[A\,,B]\right)=1.
	\end{equation}
	This will do the job since we may let $A\downarrow0$ and $B\uparrow\infty$,
	using Doob's martingale convergence theorem, to finish the proof. 
	Thus, it remains to prove \eqref{gobs}.
	
	For the remainder of the proof let us choose and fix an arbitrary space-time point
	$(t\,,x)\in(0\,,\infty)\times\R^3$ for which we plan to verify \eqref{gobs}.
	Also, let us choose an arbitrary real number $\delta>0$.
	Define
	\[
		\Pi(\delta) := \left\{ i=(i_1\,,i_2\,,i_3)\in\N^3:\ - \frac{1}{\sqrt{\delta}} \le i_\nu
		\le \frac{1}{\sqrt{\delta}}\text{\  \ for $\nu=1,2,3$}\right\},
	\]
	and
	\[
		y_i := x +(i_1,i_2,i_3)\delta
		\qquad\text{for all $i\in \Pi(\delta)$.}
	\]
	For every real number $M>0$,
	\begin{align*}
		 \P &\left(\left. 
		 	\max_{i \in \Pi(\delta)} |u(t\,,y_i)-u(t\,,x)|> 2M
			\ \right|\, \sigma(u(t\,,x)) \in [A\,,B] \right)\\
		\ge&  1 - \P \left( \left.
			\max_{i \in \Pi(\delta)} \sigma(u(t\,,x)) |Z(t\,,y_i)- Z(t\,,x)| \le 3M 
			\ \right|\,  \sigma(u(t\,,x)) \in[A\,,B] \right) \\
		& \hskip2in-\P \left(\left. \max_{i \in \Pi(\delta)} |D_t(x\,,y_i)|> M
			\ \right|\, \sigma(u(t\,,x)) \in[A\,,B]\right)\\
		=: & 1- P_1(\delta) -P_2(\delta),
	\end{align*}
	where the definition of $(P_1\,,P_2)$ is clear from context, and
	\begin{equation}\label{D_t}
		D_t(x\,,y) := u(t\,,y)-u(t\,,x) - \sigma(u(t\,,x)) \left[Z(t\,,y)- Z(t\,,x)\right],
	\end{equation}
	for every $y\in\R^3$, and 
	the mean-zero Gaussian random field $Z$ is, as before, the solution to \eqref{SHE} with $\sigma\equiv 1$.
	We are going to prove that $P_1(\delta)$ and $P_2(\delta)$
	both tend to zero as $\delta \downarrow 0$. Since $M>0$ is arbitrary, this
	will complete the proof.
	
	Consider the Gaussian process 
	\[
		\{Z(t\,,y_i)-Z(t\,,x)\}_{i\in \Pi(\delta)}.
	\]
	As was demonstrated in Section \ref{sec: additive noise}, 
	the canonical distance $d$ imposed on $\R^3$ by $Z$ satisfies
	\begin{equation}\label{d}\begin{split}
		d (i\,, j) &:= \sqrt{\E |(Z(t\,,y_i)-Z(t\,,x))-(Z(t\,,y_j)-Z(t\,,x))|^2}\\
		&\asymp\left[ \log \left(\frac{1}{\|i-j\|\delta}\right)\right]^{-(\alpha-1)/2}\,,
	\end{split}\end{equation}
	for every $i,j\in \Pi(\delta)$.
	We plan to apply a metric entropy argument in order to estimate
	the quantity on the left-hand side of \eqref{eq:DF} below. 
	
	Recall from
	Dudley \cite{Dudley} and Fernique \cite{Fernique}---see also \eqref{DF}---that
	\begin{equation}\label{eq:DF}
		\E\left(\max_{i \in \Pi(\delta)} |Z(t\,,y_i)-Z(t\,,x)|\right)
		\asymp \int_0^{\text{diam}[\Pi(\delta)]}\sqrt{\log_+\mathcal{N}(\varepsilon)}
		\,\d\varepsilon,
	\end{equation}
	where $\mathcal{N}(\varepsilon) = \mathcal{N}_{\Pi(\delta)}(\varepsilon)$ denotes the minimum number
	of $d$-balls of radius $\varepsilon>0$ that are needed to cover
	$\Pi(\delta)$ --- this is the \emph{metric
	entropy} of $\Pi(\delta)$ --- and $\text{diam}[\Pi(\delta)]$ denotes the
	diameter of $\Pi(\delta)$ in the metric $d$; that is,
	\[
		\text{diam}[\Pi(\delta)] := \max_{i,j\in \Pi(\delta)}d(i\,,j).
	\]
	It might help to also recall that the implied constants in \eqref{eq:DF}
	can be chosen to be universal and hence do not depend on the
	various parameters of our problem \cite{Dudley,Fernique}.
	
	If $i,j\in \Pi(\delta)$ satisfy $\|i-j\|=1$, then \eqref{d} assures us that
	\[
		d(i\,,j)\asymp \left[ \log (1/\delta) \right]^{-(\alpha-1)/2} =: \varepsilon_0,
	\]
	and hence,
	\[
		\mathcal{N}(\varepsilon)\asymp\vert \delta\vert^{-3/2}\qquad
		\text{uniformly for all $\varepsilon\in(0\,,\varepsilon_0)$}.
	\]
	And  a combintorial argument
	that uses only \eqref{d} shows that there exists a universal constant $c_1 \in(1\,,\infty)$ ---
	independently of $\delta$ --- such that
	\begin{equation}\label{c1}
		c_1^{-1} \delta^{3/2} \exp \left( c_1^{-1} \varepsilon^{-2/(\alpha-1)}\right)
		\le \mathcal{N}(\varepsilon) 
		\le c_1 \delta^{3/2} \exp \left( c_1 \varepsilon^{-2/(\alpha-1)}\right),
	\end{equation}
	uniformly for every $\varepsilon \ge \varepsilon_0$.
	In the same way, we can find a universal constant $c_2\in(0\,,\infty)$ ---
	independently of $\delta$ --- such that
	\begin{equation}\label{c2}
		c_2^{-1}{[\log(1/\delta)]^{(1-\alpha)/2}}\le
		\text{diam}[\Pi(\delta)] \le c_2{[\log(1/\delta)]^{(1-\alpha)/2}}.
	\end{equation}
	One can plug the results of \eqref{c1} and \eqref{c2} into \eqref{eq:DF} in order
	to deduce the following bounds:
	\begin{align*}
		&\E \left(\max_{i \in \Pi(\delta)} |Z(t\,,y_i)-Z(t\,,x)| \right)\\
		&\hskip1in\lesssim \varepsilon_0 \sqrt{ \log(1/\delta)} + 
			\int_{\varepsilon_0}^{c_2[\log (1/\delta)]^{-(\alpha-1)/2}} 
			\left[{-\log(1/\delta)}+ \frac{1}{%
			\varepsilon^{2/(\alpha-1)}}\right]^{1/2}\d\varepsilon\\
		&\hskip1in\lesssim \vert\log\delta\vert^{1-(\alpha/2)} ,
	\end{align*}
	where the parameter dependencies are, as before, uniformly over all choices
	of $(t\,,x\,,\delta)\in(0\,,\infty)\times\R^3\times(0\,,\infty)$.
	Similarly, one derives a matching lower bound, thus leading to the following:
	\begin{equation}\label{EZ-Z}
		\E\left(\max_{i \in \Pi(\delta)} |Z(t\,,y_i)-Z(t\,,x)| \right) \asymp 
		\vert\log\delta\vert^{1-(\alpha/2)} =: \varrho(t\,,\delta).
	\end{equation}
	Careful scrutiny of the parameter dependencies shows that $\varrho$
	does not depend on $x$. Because $\alpha\in (1\,,2)$,
	\begin{align}\label{E:Mtd}
		\lim_{\delta\downarrow0}\varrho(t\,,\delta)=\infty.
	\end{align}
	By the Borell, Sudakov--Tsirel'son  inequality \cite{Borell,ST}, for some $c> 0$ small enough, 
	\begin{equation}\label{CoM}
		\P \left\{  \max_{i \in \Pi(\delta)} |Z(t\,,y_i)-Z(t\,,x)|\leq \frac{c}{A} \varrho(t\,,\delta)\right\}
		\le 2 \exp \left( - \frac{c^2[\varrho(t\,,\delta)]^2}{2A^2V(t)}\right),
	\end{equation} 
	where
	 \[
	 	V(t) := \max_{i \in \Pi(\delta)} \left[\text{Var} (Z(t\,,y_i)-Z(t\,,x)) \right].
	\]
	 Owing to {Proposition \ref{pr:Z:modulus}}, 
	 \[
	 	V(t) \asymp [\log(1/\delta)]^{1-\alpha},\quad
		\text{whence}\quad
		\frac{[\varrho(t\,,\delta)]^2}{2V(t)}\asymp\log(1/\delta),
	\]
	uniformly for all $\delta>0$.
	Therefore, we apply \eqref{EZ-Z} one more time, and plug the end result in 
	\eqref{CoM} in order to see that there exists a finite constant $q>1$ such that 
	\begin{equation}\label{lambda'}
		\P \left\{  \max_{i \in \Pi(\delta)} |Z(t\,,y_i)-Z(t\,,x)|\leq \frac{c}{A}\varrho(t\,,\delta)\right\}
		\le q\delta^{1/q}\qquad\text{uniformly for all $\delta\in(0\,,1)$}.
	\end{equation} 
	Consequently, 
	\begin{align*}
		\lim_{\delta \downarrow 0}
			\P \bigg\{\max_{i \in \Pi(\delta)} \sigma(u(t\,,x)) &|Z(t\,,y_i)-Z(t\,,x)|\leq
			c\varrho(t\,,\delta) ~,~ \sigma(u(t\,,x))\in[A\,,B]\bigg\}\\
		&\hskip1in\le \lim _{\delta \downarrow 0}\P \left\{
			\max_{i \in \Pi(\delta)}|Z(t\,,y_i)-Z(t\,,x)|\leq \frac{c\varrho(t\,,\delta)}{A} \right\}\\
		&\hskip1in=0,
	\end{align*}
	thanks to \eqref{lambda'}, where $A$ is defined in \eqref{eq:A, B}.  
	It follows from this fact that
	$\lim_{\delta\downarrow0} P_1(\delta)=0.$
	\bigskip
	
	In order to complete the proof, it remains to show that
	\begin{align}\label{E:P2d->0}
		\lim_{\delta\downarrow 0} P_2(\delta)=0.
	\end{align}
	Let us recall \eqref{D_t} and write
	\begin{align*}
		\P\left\{\max_{i \in \Pi(\delta)} |D_t(x\,,y_i)|> \varrho(t\,,\delta)
			~,~ \sigma(u(t\,,x))\in [A\,,B]\right\}
		&\le\P \left\{ \max_{i \in \Pi(\delta)} |D_t(x\,,y_i)|> \varrho(t\,,\delta)\right\}\\
		&\le |\Pi(\delta)|\max_{i \in \Pi(\delta)} \P \left\{
			D_t(x\,,y_i) > \varrho(t\,,\delta) \right\}\\
		&\lesssim \delta^{-3/2}\max_{i \in \Pi(\delta)} \P \left\{
			D_t(x\,,y_i) > \varrho(t\,,\delta) \right\},
	\end{align*}
	uniformly for all sufficiently-small $\delta>0$, where $|\,\cdots|$ denotes cardinality.
	We may notice that
	\[
		D_t(x\,,y_i) = (\nabla_{\bm\varepsilon}u)(t\,,x) - \sigma(u(t\,,x))(\nabla_{\bm\varepsilon}
		Z)(t\,,x),
	\]
	for a certain $\bm\varepsilon = \bm\varepsilon(x\,,y_i)\in(0\,,\infty)^3$
	that satisfies $\|\bm\varepsilon\|\lesssim\sqrt\delta$, where the implied
	constant is universal and finite.  Therefore, Proposition \ref{pr:localization} implies that every
	random variable $|D_t(x\,,y_i)|$ is sub Gaussian. In fact, there exists
	$\lambda_0>0$, small enough, such that, for all $T\in(0\,,\infty)$,
	\[
		\sup_{s\in[0,T]}\sup_{x\in\R^3}
		\max_{i\in N(\delta)}\E\left[\exp\left\{
		\lambda_0\left[\log(1/\delta)\right]^{3(\alpha-1)/2}
		|D_s(x\,,y_i)|^2 \right\}\right]\lesssim1,    
	\]
	uniformly for all sufficiently-small $\delta>0$.
	
	Therefore, by Chebyshev's inequality and \eqref{EZ-Z},
	{\begin{align*}
		\P\left\{|D_t(x\,,y_i)| > \varrho(t\,,\delta) \right\} &\lesssim 
			\exp \left( -\lambda_0\left[\log(1/\delta)\right]^{3(\alpha-1)/2}
			[\varrho(t\,,\delta)]^2\right)\\  
		& \le \exp \left(- C\left[\log(1/\delta)\right]^{(\alpha+ 1)/2} \right)    
	\end{align*} }
	for a finite constant $C>0$ that depends only on $t\in[0\,,T]$. 
	Since the cardinality of $\Pi(\delta)$ satisfies $|\Pi(\delta)| \asymp \delta^{-3/2}$,
	it follows from the preceding displayed inequality that
	\[
		\P \left\{\max_{i \in \Pi(\delta)} |D_t(x,y_i)|> \varrho(t\,,\delta)\right\}
		\lesssim 
		{\exp \left( - C \left[\log(1/\delta)\right]^{(\alpha+1)/2}
		+ \tfrac{3}{2} \log (1/\delta)\right), }
	\]
	which tends to $0$ as $\delta \to 0$.  A scaling argument 
implies that
	\[
		\lim_{\delta\downarrow0}
		\P \left\{\max_{i \in \Pi(\delta)} |D_t(x,y_i)|> \lambda\varrho(t\,,\delta)\right\}=0
		\qquad\text{for every fixed $\lambda>0$}.
	\]
	In particular, for all $\lambda>0$,
	\begin{align*}
		&\P \left\{\left.
			\max_{i \in \Pi(\delta)} |u(t\,,y_i)-u(t\,,x)| < \lambda\varrho(t\,,\delta) 
			\ \right|\, \sigma(u(t\,,x)) \in[A\,,B] \right\}\\
		&\hskip0.8in\le \P \left\{\left.
			\max_{i \in \Pi(\delta)} |Z(t\,,y_i)-Z(t\,,x)| < {\frac{2\lambda}{A}}\varrho(t\,,\delta) 
			\ \right|\, \sigma(u(t\,,x)) \in[A\,,B] \right\}\\
		&\hskip3.6in + \P \left\{\max_{i \in \Pi(\delta)} |D_t(x,y_i)|> \lambda\varrho(t\,,\delta)\right\}\\
		&\hskip0.8in\le \P \left\{\left.
			\max_{i \in \Pi(\delta)} |Z(t\,,y_i)-Z(t\,,x)| < {\frac{2\lambda}{A}}\varrho(t\,,\delta) 
			\ \right|\, \sigma(u(t\,,x)) \in[A\,,B] \right\} + o(1),
	\end{align*}
	as $\delta\downarrow0$. Thanks to \eqref{EZ-Z} and Proposition \ref{pr:max:Z}, we can
	choose {$\lambda:=\lambda(A)$} small enough to ensure that the right-most
	probability above also tends to zero as $\delta\downarrow0$.
	In light of \eqref{E:Mtd}, this verifies \eqref{gobs}
	and hence concludes the proof of Theorem \ref{th:conditional}.	
\end{proof}

The dervivation of Theorem \ref{th:conditional} admittedly required some effort. But now
we can adjust that derivation --- without a great deal of additional effort --- in order to  verify
Theorem \ref{th:local}.

\begin{proof}[Proof of Theorem \ref{th:local} (sketch)]
	The conditioning in the equivalent statement
	\eqref{cond} to Theorem \ref{th:conditional} arose 
	because, during the course of the proof of Theorem \ref{th:conditional}, we needed 
	to prove that 
	\begin{equation}\label{last}
		\lim_{\delta \to 0}
		\P \left\{\max_{i \in \Pi(\delta)} |\sigma(u(t\,,x))|\cdot |Z(t\,,y_i)-Z(t\,,x)|
		\le\lambda\varrho(t\,,\delta) \right\}=0,
	\end{equation}
	for a suitably-small choice of $\lambda>0$,
	and $(t\,,x)\mapsto\sigma(u(t\,,x))$ 
	could, in principle, be frequently close to --- or possibly even equal to --- zero.
	In the present setting however, $\sigma$ is bounded uniformly from below, away from zero.
	Therefore, in the present setting, \eqref{last} follows immediately Proposition
	\ref{pr:max:Z}, as long as $\lambda$ is a small-enough [but otherwise fixed]
	positive constant. The remainder
	of the proof of Theorem \ref{th:conditional} remains essentially intact.
\end{proof}

\begin{small}
\bigskip

\noindent\textbf{Le Chen}, \textbf {Jingyu Huang}, \textbf{Davar Khoshnevisan}, and \textbf{Kunwoo Kim}\\

\noindent 	Department of Mathematics, University of Kansas, 
                 Lawrence, KS, 66044 \\       
\noindent Department of Mathematics, University of Utah,
		Salt Lake City, UT 84112-0090\\  
\noindent 	Department of Mathematics, University of Utah,
		Salt Lake City, UT 84112-0090\\ 
\noindent  	Department of Mathematics, Pohang University of Science and Technology, Pohang, Gyeongbuk, Korea 37673 \\      
\noindent\emph{Emails} \& \emph{URLs}:\\
        \indent\texttt{chenle@ku.edu}\hfill\hfil
	        \url{http://www.math.ku.edu/u/chenle/}\\
        \indent\texttt{jhuang@math.utah.edu}\hfill\hfil
        	\url{http://www.math.utah.edu/~jhuang/}\\
	\indent\texttt{davar@math.utah.edu}\hfill\hfil
		\url{http://www.math.utah.edu/~davar/}\\
	\indent\texttt{kunwoo@postech.ac.kr}\hfill\hfil
		\url{http://math.postech.ac.kr/~kunwoo/}
\end{small}

\end{document}